\documentclass[oneside,reqno,12pt]{amsart}%
\usepackage[left=1in,top=1in,right=1in,bottom=1in]{geometry}
\usepackage{amscd}
\usepackage{amsfonts}
\usepackage{graphicx}
\usepackage{amsmath}
\usepackage{amssymb}
\usepackage{amsthm}
\usepackage{latexsym}

\providecommand{\U}[1]{\protect\rule{.1in}{.1in}}
\newtheorem{theorem}{Theorem}[section]

\newtheorem{corollary}[theorem]{Corollary}

\newtheorem{definition}[theorem]{Definition}
\newtheorem{example}[theorem]{Example}

\newtheorem{lemma}[theorem]{Lemma}

\newtheorem{proposition}[theorem]{Proposition}
\newtheorem{remark}[theorem]{Remark}

\newtheorem{question}[theorem]{Question}
\newcommand{\note}[2][\null]{%
  \marginpar{\renewcommand{\baselinestretch}{1}\vspace{-1em}\hrule\vspace{3pt}%
  \tiny\raggedright{#2\ifx#1\null\else\\\hfill---
  {\em #1}\fi}\vspace{1.5em}}%
}

\begin{document}
\title[Plancherel Measure and Admissibility]{Regular Representations of Time-Frequency Groups}
\author[A. Mayeli]{Azita Mayeli}
\address{Mathematics Department, Queensborough C. College,  City University of New
York, Bayside, NY 11362 U.S.A}
\email{amayeli@qcc.cuny.edu}
\author[V. Oussa]{Vignon Oussa}
\address{Dept.\ of Mathematics \& Computer Science\\
Bridgewater State University\\
Bridgewater, MA 02325 U.S.A.\\
 }
\curraddr{131 Summer Street, Bridgewater, Massachusetts 02325}
\email{vignon.oussa@bridgew.edu}
\thanks{}

\begin{abstract}
In this paper, we study the Plancherel measure  of a class of non-connected
nilpotent groups which is of special interest in Gabor theory. Let $G$ be a
time-frequency group. More precisely, that is $G=\left\langle
T_{k},M_{l}:k\in%
%TCIMACRO{\U{2124} }%
%BeginExpansion
\mathbb{Z}
%EndExpansion
^{d},l\in B%
%TCIMACRO{\U{2124} }%
%BeginExpansion
\mathbb{Z}
%EndExpansion
^{d}\right\rangle ,$ $T_{k}$, $M_{l}$ are translations and modulations
operators acting in $L^{2}(\mathbb{R}^{d}),$ and $B$ is a non-singular matrix.
We compute the Plancherel measure of the left regular representation of
$G\ $which is denoted by $L.$ The action of $G$ on $L^{2}(\mathbb{R}^{d})$
induces a representation which we call a Gabor representation. Motivated by
the admissibility of this representation, we compute the decomposition of $L$ into direct integral of irreducible representations by providing a precise description of the unitary dual and its Plancherel
measure. As a result, we generalize Hartmut F\"uhr's results which are only
obtained for the restricted case where $d=1$, $B=1/L,L\in%
%TCIMACRO{\U{2124} }%
%BeginExpansion
\mathbb{Z}
%EndExpansion
$ and $L>1.$ Even in the case where $G$ is not type I, we are able to obtain a
decomposition of the left regular representation of $G$ into a direct integral
decomposition of irreducible representations when $d=1$. Some interesting
applications to Gabor theory are given as well. For example, when $B$ is an
integral matrix, we are able to obtain a direct integral decomposition of the
Gabor representation of $G.$

\end{abstract}
\maketitle

\date{\today}

%\title[short text for running head]{full title}

%Only \author and \address are required; other information is
%optional.  Remove any unused author tags.

%author one information
%\author[short version for running head]{name for top of paper}

%The 2010 edition of the Mathematics Subject Classification is
%now available.  If you are citing a classification from the
%new scheme, use the following input coding instead.
%\subjclass[2010]{Primary }

%Text of article.

%Bibliographies can be prepared with BibTeX using amsplain,
%amsalpha, or (for "historical" overviews) natbib style.
%Insert the bibliography data here.

\section{Introduction}

Let $G$ be a locally compact group with type $I$ left regular representation.
The Plancherel theorem guarantees the existence of a measure $\mu$ on the
unitary dual of $G\ $such that once a Haar measure is fixed on the group $G$,
$\mu$ is uniquely determined. Although the existence of the Plancherel measure
is given; it is generally a hard but interesting problem to compute it. Let
$\widehat{G}$ be the unitary dual of the group $G.$ The group Fourier
transform $f\mapsto\widehat{f}$ maps $L^{1}\left(  G\right)  \cap L^{2}\left(
G\right)  $ into $\int_{\widehat{G}}^{\oplus}\mathcal{H}_{\pi}\otimes
\mathcal{H}_{\pi}d\mu\left(  \pi\right)  .$ Also, the group Fourier transform
extends to a unitary map defined on $L^{2}\left(  G\right)  $. This extension
is known as the Plancherel transform $\mathcal{P}$. For arbitrary functions
$f$ and $g\in L^{2}\left(  G\right)  ,$ the following holds true:
\[
\left\langle f,g\right\rangle _{L^{2}\left(  G\right)  }=\int_{\widehat{G}%
}\mathrm{trace}\left[  \widehat{f}\left(  \pi\right)  \widehat{g}\left(
\pi\right)  ^{\ast}\right]  d\mu\left(  \pi\right)  .
\]
Moreover, if $L$ is the left regular representation of $G$ acting in
$L^{2}\left(G\right)  $ then we obtain a \textbf{unique} direct integral
decomposition of $L$ into unitary irreducible representations of $G$ such that%
\[
\mathcal{P\circ}\text{ }L\circ\mathcal{P}^{-1}=\int_{\widehat{G}}^{\oplus}%
\pi\otimes1_{\mathcal{H}_{\pi}}\text{ }d\mu\left(  \pi\right)  .
\]

Let $B$ be an invertible matrix of order $d.$ For $k$ $\in%
%TCIMACRO{\U{2124} }%
%BeginExpansion
\mathbb{Z}
%EndExpansion
^{d}$, and $l\in B%
%TCIMACRO{\U{2124} }%
%BeginExpansion
\mathbb{Z}
%EndExpansion
^{d},$ we define the following unitary operators $T_{k},M_{l}:L^{2}\left(
%TCIMACRO{\U{211d} }%
%BeginExpansion
\mathbb{R}
%EndExpansion
^{d}\right)  \rightarrow L^{2}\left(
%TCIMACRO{\U{211d} }%
%BeginExpansion
\mathbb{R}
%EndExpansion
^{d}\right)  $ such that $T_{k}f\left(  x\right)  =f\left(  x-k\right)  $ and
$M_{l}f\left(  x\right)  =e^{2\pi i\left\langle l,x\right\rangle }f\left(
x\right)  .$ Let $G$ be the group generated by $T_{k},$ and $M_{l}$. $T_{k}$
is a time-shift operator, and $M_{l}$ is a modulation operator (or frequency
shift operator). We write
\begin{equation}
G=\left\langle T_{k},M_{l}\text{ }|\text{ }k\in%
%TCIMACRO{\U{2124} }%
%BeginExpansion
\mathbb{Z}
%EndExpansion
^{d},l\in B%
%TCIMACRO{\U{2124} }%
%BeginExpansion
\mathbb{Z}
%EndExpansion
^{d}\right\rangle . \label{G}%
\end{equation}

Characteristics of the closure of the linear span of orbits of the type $G(S)$ where $S\subset L^2(\mathbb{R}^d)$ have been studied in \cite{Bownik} when $B$ only has rational entries. Also a thorough presentation
of the theory of time-frequency analysis is given in \cite{Grog}. In
this paper, we are mainly interested in the following questions.

\begin{question}
If $G$ is type I, can we provide a description of the unitary dual of $G,$ and
a precise formula for the Plancherel measure of $G?$
\end{question}

\begin{question}
If $G$ is not type I, can we obtain a decomposition of the left regular
representation into unitary irreducible representations of $G$? Is it possible to provide a central decomposition of its left regular representation?
\end{question}

For obvious reasons, we term the group
$G$ a \textbf{time-frequency group}. There are three cases to consider.
First, it is easy to see that $G$ is a commutative discrete group if and
only if $B$ is an integral matrix. In that case, all irreducible
representations of $G$ are characters, and thanks to the Pontrjagin duality,
the Plancherel measure is well-understood. In fact if $B$ is a matrix with
integral entries, the Plancherel measure of $G$ is supported on a measurable
fundamental domain of the lattice $%
%TCIMACRO{\U{2124} }%
%BeginExpansion
\mathbb{Z}
%EndExpansion
^{d}\times\left(  B^{-1}\right)  ^{tr}%
%TCIMACRO{\U{2124} }%
%BeginExpansion
\mathbb{Z}
%EndExpansion
^{d};$ namely the set $
\left[  0,1\right)  ^{d}\times\left(  B^{-1}\right)  ^{tr}\left[  0,1\right)
^{d}.$ Interestingly, it can be shown that the Gabor representation $\mathbb{Z}^{d}\times B\mathbb{Z}^{d}\ni\left(  k,l\right)  \mapsto T_{k}\circ M_{l}$ is unitarily equivalent with a subrepresentation of the left regular
representation if and only if $\left\vert \det B\right\vert \leq1.$ Otherwise,
the Gabor representation is equivalent to a direct sum of regular
representations. Although the previous statement is not technically new, the
proof given here is based on the representation theory of the time-frequency group. Secondly, if $B$ has some rational entries, some of them non-integer, then
$G$ is a non-commutative discrete type I group. Using well-known techniques
developed by Mackey, and later on by Kleppner and Lipsman in \cite{Lipsman1},
precise descriptions of the unitary dual of $G$ and its Plancherel measure are
obtained. Two main ingredients are required to compute the Plancherel measure
of $G$. Namely, a closed normal subgroup of $N$ whose left regular
representation is type I, and a family of subgroups of $G/N$ known as the
`little groups'. We will show that the Plancherel measure in the case where
$B$ has some non integral rational entries but no irrational entry, is a fiber
measure supported on a fiber space with base space: the unitary dual of the
commutator subgroup of $G$. That is, the base space is $\widehat{\left[
G,G\right]  }.$ Using some procedure given in \cite{Unified}, we can construct
a non-singular matrix $A,$ such that $A%
%TCIMACRO{\U{2124} }%
%BeginExpansion
\mathbb{Z}
%EndExpansion
^{d}=%
%TCIMACRO{\U{2124} }%
%BeginExpansion
\mathbb{Z}
%EndExpansion
^{d}\cap\left(  B^{-1}\right)  ^{tr}%
%TCIMACRO{\U{2124} }%
%BeginExpansion
\mathbb{Z}
%EndExpansion
^{d}$ and we show that each fiber can be identified with some compact set
\[
\Lambda_{1}\times\mathbf{E}_{\sigma}\times\left\{  \chi_{\sigma}\right\}
\subset%
%TCIMACRO{\U{211d} }%
%BeginExpansion
\mathbb{R}
%EndExpansion
^{d}\times%
%TCIMACRO{\U{211d} }%
%BeginExpansion
\mathbb{R}
%EndExpansion
^{d}\times\left\{  \chi_{\sigma}\right\}
\]
where $\chi_{\sigma}\in\widehat{\left[  G,G\right]  },$ $\Lambda_{1}$ is a
measurable fundamental domain of $\left(  A^{-1}\right)  ^{tr}%
%TCIMACRO{\U{2124} }%
%BeginExpansion
\mathbb{Z}
%EndExpansion
^{d}$ and $\mathbf{E}_{\sigma}$ is the cross-section of the action of a little
group (which is a finite group here) in $%
%TCIMACRO{\U{211d} }%
%BeginExpansion
\mathbb{R}
%EndExpansion
^{d}/\left(  B^{-1}\right)  ^{tr}%
%TCIMACRO{\U{2124} }%
%BeginExpansion
\mathbb{Z}
%EndExpansion
^{d}.$ We also show that all irreducible representations of $G$ are monomial
representations modelled as acting in some finite-dimensional Hilbert spaces
with dimensions bounded above by $\left\vert \det A\right\vert $. It is worth
noticing that Hartmut F\"{u}hr has already computed the Plancherel measure of
the simplest example (Section $5.5$ of \cite{hartmut}) of the class of
groups considered in our paper. In his example, $d=1$ and $B%
%TCIMACRO{\U{2124} }%
%BeginExpansion
\mathbb{Z}
%EndExpansion
^{d}=1/L%
%TCIMACRO{\U{2124} }%
%BeginExpansion
\mathbb{Z}
%EndExpansion
$ where $L$ is some positive integer greater than one. For the more general
case in which we are interested, we obtain a parametrization of the unitary
dual of $G,$ and we derive a precise formula for the Plancherel measure. In
the case where $G$ is not type I ($B$ has some irrational entries);
unfortunately the `Mackey machine' fails. In the particular case
where $d=1,$%
\[
G=\left\langle T_{k},M_{l}\text{ }|\text{ }k\in%
%TCIMACRO{\U{2124} }%
%BeginExpansion
\mathbb{Z}
%EndExpansion
,l\in\alpha%
%TCIMACRO{\U{2124} }%
%BeginExpansion
\mathbb{Z}
%EndExpansion
\right\rangle ,\text{ }\alpha\in%
%TCIMACRO{\U{211d} }%
%BeginExpansion
\mathbb{R}
%EndExpansion
-%
%TCIMACRO{\U{211a} }%
%BeginExpansion
\mathbb{Q}
%EndExpansion
,
\]
we are able to obtain a central decomposition of the left regular representation of $L$ as well as a direct integral decomposition of the left
regular representation of
\[
G\cong\Gamma=\left\{  \left[
\begin{array}
[c]{ccc}%
1 & m & l\\
0 & 1 & k\\
0 & 0 & 1
\end{array}
\right]  :\left(  m,k,l\right)  \in%
%TCIMACRO{\U{2124} }%
%BeginExpansion
\mathbb{Z}
%EndExpansion
^{3}\right\}  
\]
into its irreducible components. In fact, we derive that the left regular representation of $G$ is unitarily
equivalent to
\begin{equation}
\int_{\left[  -1,1\right]  }^{\oplus}\int_{\left[  0,\left\vert \lambda
\right\vert \right)  }^{\oplus}\mathrm{Ind}_{K}^{\Gamma}\left(  \chi_{\left(
\left\vert \lambda\right\vert ,t\right)  }\right)  \left\vert \lambda
\right\vert dtd\lambda\label{irrational}%
\end{equation}
acting in $\int_{\left[  -1,1\right]  }^{\oplus}\int_{\left[  0,\left\vert
\lambda\right\vert \right)  }^{\oplus}l^{2}\left(
%TCIMACRO{\U{2124} }%
%BeginExpansion
\mathbb{Z}
%EndExpansion
\right)  \left\vert \lambda\right\vert dtd\lambda$ where
\[
K=\left\{  \left[
\begin{array}
[c]{ccc}%
1 & 0 & l\\
0 & 1 & k\\
0 & 0 & 1
\end{array}
\right]  :\left(  k,l\right)  \in%
%TCIMACRO{\U{2124} }%
%BeginExpansion
\mathbb{Z}
%EndExpansion
^{2}\right\}
\]
and
\[
\chi_{\left(  \left\vert \lambda\right\vert ,t\right)  }\left(  \left[
\begin{array}
[c]{ccc}%
1 & 0 & l\\
0 & 1 & k\\
0 & 0 & 1
\end{array}
\right]  \right)  =\exp\left(  2\pi i\left(  \left\vert \lambda\right\vert
l+tk\right)  \right)  .
\]
To the best of our knowledge, the decomposition given in (\ref{irrational}) and the central decomposition obtained in Theorem \ref{central} are both
appearing for the first time in the literature.

\section*{Acknowledgements} 
 Sincerest thanks go to B. Currey, for providing support and suggestions during the preparation of this work. We also thank the anonymous referees. Their suggestions and corrections were crucial to the improvement of our work.

\section{Preliminaries}
Let us start by setting up some notation. Given a matrix $A$ of order $d,$
$A^{tr}$ stands for the transpose of $A,$ and $A^{-tr}=\left(  A^{-1}\right)
^{tr}$ stands for the inverse transpose of $A.$ Let $F$ be a field or a ring.
It is standard to use $GL\left(  d,F\right)  $ to denote the set of
invertible matrices of order $d$ with entries in $F.$ Let $G$ be a locally
compact group. The unitary dual, which is the set of all irreducible unitary
representations of $G$ is denoted by $\widehat{G}.$ Given $x\in%
%TCIMACRO{\U{211d} }%
%BeginExpansion
\mathbb{R}
%EndExpansion
^{d},$ we define a character which is a one-dimensional unitary representation
of $%
%TCIMACRO{\U{211d} }%
%BeginExpansion
\mathbb{R}
%EndExpansion
^{d}$ into the one-dimensional torus as $\chi_{x}:%
%TCIMACRO{\U{211d} }%
%BeginExpansion
\mathbb{R}
%EndExpansion
^{d}\rightarrow\mathbb{T}$ where $\chi_{x}\left(  y\right)  =e^{2\pi
i\left\langle x,y\right\rangle }.$ Let $H$ be a subgroup of $G.$ The index of
$H$ in $G$ is denoted by $\left(  G:H\right)  =\left\vert G/H\right\vert .$ We
will use $\mathbf{1}$ to denote the identity operator acting in some Hilbert
space. Given two isomorphic groups $G,H$ we write $G\cong H.$

The reader who is not familiar with the theory of direct integrals is invited
to refer to \cite{Folland}.

\begin{definition}
Given a countable sequence $\left\{  f_{i}\right\}  _{i\in I}$ of vectors in
a Hilbert space $\mathcal{H},$ we say $\left\{  f_{i}\right\}  _{i\in I}$
forms a \textbf{frame} if and only if there exist strictly positive real
numbers $A,B$ such that for any vector $f\in\mathcal{H}$
\[
A\left\Vert f\right\Vert ^{2}\leq\sum_{i\in I}\left\vert \left\langle
f,f_{i}\right\rangle \right\vert ^{2}\leq B\left\Vert f\right\Vert ^{2}.
\]
In the case where $A=B$, the sequence of vectors $\left\{  f_{i}\right\}
_{i\in I}$ is called a \textbf{tight frame}, and if $A=B=1$, $\left\{
f_{i}\right\}  _{i\in I}$ is called a \textbf{Parseval frame}.
\end{definition}

\begin{remark}
If $\left\{  f_{i}\right\}  _{i\in I}$ is a \textbf{Parseval frame} such that
for all $i\in I,\left\Vert f_{i}\right\Vert =1$, then $\left\{  f_{i}\right\}
_{i\in I}$ is an orthonormal basis for $\mathcal{H}$.
\end{remark}

\begin{definition}
A lattice $\Lambda$ in $\mathbb{R}^{d}$ is a discrete additive subgroup of
$\mathbb{R}^{d}$. In other words, every point in $\Lambda$ is isolated and
$\Lambda=A\mathbb{Z}^{d}$ for some matrix $A$. We say $\Lambda$ is a full rank
lattice if $A$ is nonsingular, and we denote the dual of $\Lambda$ by
$\left(A^{-1}\right)^{tr}\Lambda.$ A \textbf{fundamental domain (or transveral)} $D$ for a
lattice in $\mathbb{R}^{d}$ is a measurable set such that the following hold:
\[
(D+l)\cap(D+l^{\prime})\neq\emptyset
\]
for distinct $l,$ $l^{\prime}$ in $\Lambda,$ and $\mathbb{R}^{d}%
={\bigcup\limits_{l\in\Lambda}}\left(  D+l\right)  .$ 
\end{definition}

\begin{definition}
Let $\Lambda=A\mathbb{Z}^{d}\times B\mathbb{Z}^{d}$ be a full rank lattice in
$\mathbb{R}^{2d}$ and $g\in L^{2}\left(  \mathbb{R}^{d}\right)  $. The family
of functions in $L^{2}\left(  \mathbb{R}^{d}\right)  $,
\begin{equation}
\mathcal{G}\left(  g,A\mathbb{Z}^{d}\times B\mathbb{Z}^{d}\right)  =\left\{
e^{2\pi i\left\langle k,x\right\rangle }g\left(  x-n\right)  :k\in
B\mathbb{Z}^{d},n\in A\mathbb{Z}^{d}\right\}  \label{Gabor}%
\end{equation}
is called a \textbf{Gabor system}.
\end{definition}

\begin{definition}
Let $m$ be the Lebesgue measure on $\mathbb{R}^{d}$, and consider a full rank
lattice $\Lambda=A\mathbb{Z}^{d}$ inside $\mathbb{R}^{d}$.

\begin{enumerate}
\item The \textbf{volume} of $\Lambda$ is defined as $vol\left(
\Lambda\right)  =m\left(  \mathbb{R}^{d}/\Lambda\right)  =\left\vert \det
A\right\vert .$

\item The \textbf{density} of $\Lambda$ is defined as $d\left(  \Lambda
\right)  =\dfrac{1}{\left\vert \det A\right\vert }.$
\end{enumerate}
\end{definition}

\begin{lemma}
\label{density}Given a separable full rank lattice $\Lambda=A\mathbb{Z}%
^{d}\times B\mathbb{Z}^{d}$ in $\mathbb{R}^{2d}$. The following statements are equivalent
\end{lemma}

\begin{enumerate}
\item There exits $f\in L^{2}(\mathbb{R}^{d})$ such that $\mathcal{G}\left(
f,A\mathbb{Z}^{d}\times B\mathbb{Z}^{d}\right)  $ is a Parseval frame in
$L^{2}\left(  \mathbb{R}^{d}\right)  .$

\item $vol\left(  \Lambda\right)  =\left\vert \det A\det B\right\vert \leq1.$

\item There exists $f\in L^{2}\left(  \mathbb{R}^{d}\right)  $ such that
$\mathcal{G}\left(  f,A\mathbb{Z}^{d}\times B\mathbb{Z}^{d}\right)  $ is
complete in $L^{2}\left(  \mathbb{R}^{d}\right)  .$
\end{enumerate}

For a proof of the above lemma, we refer the reader to Theorem $3.3$ in \cite{Han}.

\begin{lemma}
\label{volume} Let $\Lambda$ be a full-rank lattice in $\mathbb{R}^{2}$. There
exists a function $f\in L^{2}\left(  \mathbb{R}^{d}\right)  $ such that
$\mathcal{G}\left(  f,\Lambda\right)  $ is an orthonormal basis if and only if
$vol\left(  \Lambda\right)  =1.$ Also, if $\mathcal{G}\left(  f,\Lambda
\right)  $ is a Parseval frame for $L^{2}(\mathbb{R}^{d})$ then $\Vert
f\Vert^{2}=vol(\Lambda).$
\end{lemma}
A proof of the Lemma \ref{volume} is given in \cite{Han}. Now, let $
\mathbb{F}=\left\langle T_{x},M_{y}\text{ }|\text{ }x\in%
%TCIMACRO{\U{211d} }%
%BeginExpansion
\mathbb{R}
%EndExpansion
^{d},y\in%
%TCIMACRO{\U{211d} }%
%BeginExpansion
\mathbb{R}
%EndExpansion
^{d}\right\rangle$ and let $f$ be a square-integrable function over $%
%TCIMACRO{\U{211d} }%
%BeginExpansion
\mathbb{R}
%EndExpansion
^{d}.$ It is easy to see that%
\begin{equation}
T_{x}M_{y}T_{x}^{-1}M_{y}^{-1}f=e^{-2\pi i\left\langle y,x\right\rangle }f,
\label{action}%
\end{equation}
and $e^{-2\pi i\left\langle y,x\right\rangle }$ is a central element of the
group $\mathbb{F}$. Thus $\mathbb{F}$ is a non-commutative, connected,
\textbf{non-simply} connected two-step nilpotent Lie group. In fact
$\mathbb{F}$ is isomorphic to the \textbf{reduced} $2d+1$ dimensional
Heisenberg group. The $2d+1$-dimensional Heisenberg group has Lie algebra
\[
\mathfrak{h}=%
%TCIMACRO{\U{211d} }%
%BeginExpansion
\mathbb{R}
%EndExpansion
\text{-span }\left\{  X_{1},\cdots,X_{d,}Y_{1},\cdots,Y_{d},Z\right\}
\]
with non-trivial Lie brackets $\left[  X_{i},Y_{i}\right]  =Z.$ Let
\begin{equation}
\mathbb{H}=\exp\left(  \mathfrak{h}\right)  \label{Heisenberg}%
\end{equation}
Clearly, $\mathbb{H}$ is a simply connected, connected $2n+1$-dimensional
Heisenberg group, and $\exp\left(
%TCIMACRO{\U{2124} }%
%BeginExpansion
\mathbb{Z}
%EndExpansion
Z\right)  $ is a discrete central subgroup of $\mathbb{H}$. Moreover,
$\mathbb{H}$ is the universal covering group of $\mathbb{F}$. That is,
$\mathbb{F\ }$is isomorphic to the group $\mathbb{H}/\exp\left(
%TCIMACRO{\U{2124} }%
%BeginExpansion
\mathbb{Z}
%EndExpansion
Z\right)  .$

We will now provide a light introduction to the notion of admissibility of
unitary representations. A more thorough exposure to the theory is given in
\cite{hartmut}.  However, we will discuss part of
the material which is necessary to fully understand the results obtained in
our work. Let $\pi$ be a unitary representation of a locally compact
group $X,$ acting in some Hilbert space $H.$ We say that $\pi$ is \textbf{admissible}, if and only if there
exists some vector $\phi\in H$ such that the operator $W_{\phi}$
\[
W_{\phi}:H\rightarrow L^{2}\left(  X\right)  \text{ where }W_{\phi}\psi\left(
x\right)  =\left\langle \psi,\pi\left(  x\right)  \phi\right\rangle 
\]
defines an isometry from $H$ to $L^2\left(X\right)$. Let $\left(  \rho,\mathcal{K}\right)  $ be an
arbitrary type I representation of the group $
G=\left\langle T_{k},M_{l}:k\in\mathbb{Z}^{d},l\in B\mathbb{Z}^{d}\right\rangle .$ 
Moreover, let us suppose that $X$ is a type I unimodular group. Let us also
suppose that we are able to obtain a direct integral decomposition of $\rho$
as follows%
\[
\rho\cong\int_{\widehat{X}}^{\oplus}\left(  \oplus_{k=1}^{n_{\pi}}\pi\right)
\text{ }d\overline{\mu}\left(  \pi\right) 
\]
where $d\overline{\mu}$ is a measure defined on the unitary dual of $X.$ According to well-known theorems developed in \cite{hartmut}; $\left(
\rho,\mathcal{K}\right)  $ is admissible if and only if it is unitarily
equivalent with a subrepresentation of the left regular representation, and
the multiplicity function is integrable over the spectrum $\rho$. That
is:\ $\overline{\mu}$ is absolutely continuous with respect to the Plancherel
measure $\mu$ supported on $\widehat{X}$ and $\int_{\widehat{X}}n_{\pi
}d\overline{\mu}\left(  \pi\right)  <\infty.$ If a representation $\rho$ is
\textbf{admissible}, in theory it is known (see \cite{hartmut}) how to
construct all \textbf{admissible vectors}. Let us describe such process in
general terms. First, we must construct a unitary operator
\[
U:\mathcal{K}\rightarrow\int_{\widehat{X}}^{\oplus}\left(  \oplus
_{k=1}^{n_{\pi}}\mathcal{H}_{\pi}\right)  \text{ }d\overline{\mu}\left(
\pi\right)
\]
intertwining the representation $\rho$ with $\int_{\widehat{X}}^{\oplus}%
n_{\pi}\pi$ $d\overline{\mu}\left(  \pi\right)  .$ Next, we define a
measurable field $\left(  \mathbf{F}_{\lambda}\right)  _{\lambda\in
\widehat{X}}$ of operators in $\int_{\widehat{X}}^{\oplus}\oplus_{k=1}%
^{n_{\pi}}\mathcal{H}_{\pi}$ $d\overline{\mu}\left(  \pi\right)  $ such that
each operator $\mathbf{F}_{\lambda}$ is an isometry. All admissible vectors
are of the type $U^{-1}\left(  \left(  \mathbf{F}_{\lambda}\right)  _{\lambda\in
\widehat{X}}\right)  .$

In the remainder of the paper, we will focus on time-frequency  groups.
We will compute the Plancherel measure of the group whenever $G$ is type I,
and we will obtain a direct integral decomposition of the left regular
representation if $G$ is not type I and $d=1$. Some application to Gabor
theory will be discussed throughout the paper as well.

\section{Normal Subgroups of $G$}

In this section, we will study the structure of normal subgroups of the time frequency group. The reason why this section is important, is because
part of what the Mackey machine \cite{Lipsman2} needs is an explicit
description of the unitary dual of type I normal subgroups in order to compute
the unitary dual of the whole group. In this paper, unless we state otherwise,
$G$ stands for the following group:
\[
\left\langle T_{k},M_{l}\text{ }|\text{ }k\in%
%TCIMACRO{\U{2124} }%
%BeginExpansion
\mathbb{Z}
%EndExpansion
^{d},l\in B%
%TCIMACRO{\U{2124} }%
%BeginExpansion
\mathbb{Z}
%EndExpansion
^{d}\right\rangle .
\]
We recall that the subgroup generated by operators of the type $T_{k}\circ
M_{l}\circ T_{k}^{-1}\circ M_{l}^{-1}$ is called the commutator subgroup of
$G$ and is denoted $\left[  G,G\right]  $. From now on, to simplify the
notation, we will simply omit the symbol $\circ$ whenever we are composing operators.

\begin{lemma}
\label{GG}$\left[  G,G\right]  $ is isomorphic to a subgroup of the torus.

\begin{enumerate}
\item If $B\in GL\left(  d,%
%TCIMACRO{\U{2124} }%
%BeginExpansion
\mathbb{Z}
%EndExpansion
\right)  $ then $G$ is commutative and isomorphic to $%
%TCIMACRO{\U{2124} }%
%BeginExpansion
\mathbb{Z}
%EndExpansion
^{d}\times B%
%TCIMACRO{\U{2124} }%
%BeginExpansion
\mathbb{Z}
%EndExpansion
^{d}$.

\item If $B\ $is in $GL\left(  d,%
%TCIMACRO{\U{211a} }%
%BeginExpansion
\mathbb{Q}
%EndExpansion
\right)  -GL\left(  d,%
%TCIMACRO{\U{2124} }%
%BeginExpansion
\mathbb{Z}
%EndExpansion
\right)  $ then $\left[  G,G\right]  $ is a central subgroup of $G$, and is a
cyclic group$.$ $G$ is not commutative but it is a type I discrete unimodular group.

\item If $B\in GL\left(  d,%
%TCIMACRO{\U{211d} }%
%BeginExpansion
\mathbb{R}
%EndExpansion
\right)  -GL\left(  d,%
%TCIMACRO{\U{211a} }%
%BeginExpansion
\mathbb{Q}
%EndExpansion
\right)  $ then, $G$ is a non-commutative two-step nilpotent group, and its
commutator subgroup is a infinite subgroup of the circle group.
\end{enumerate}
\end{lemma}

\begin{proof}
To show part (a), let $l\in B%
%TCIMACRO{\U{2124} }%
%BeginExpansion
\mathbb{Z}
%EndExpansion
^{d},k\in%
%TCIMACRO{\U{2124} }%
%BeginExpansion
\mathbb{Z}
%EndExpansion
^{d}$ where $B$ is a non-invertible matrix. Let $f\in L^{2}\left(
%TCIMACRO{\U{211d} }%
%BeginExpansion
\mathbb{R}
%EndExpansion
^{d}\right)  .$ We have
\[
T_{k}M_{l}T_{k}^{-1}M_{l}^{-1}f=e^{-2\pi i\left\langle l,k\right\rangle
}f=\chi_{l}\left(  k\right)  f.
\]
If $B\in GL\left(  d,%
%TCIMACRO{\U{2124} }%
%BeginExpansion
\mathbb{Z}
%EndExpansion
\right)  $ then the commutator subgroup of $G$ is trivial, and $G$ is an
abelian group isomorphic to $%
%TCIMACRO{\U{2124} }%
%BeginExpansion
\mathbb{Z}
%EndExpansion
^{d}\times B%
%TCIMACRO{\U{2124} }%
%BeginExpansion
\mathbb{Z}
%EndExpansion
^{d}.$ For part (b), let us suppose that $B\in GL\left(  d,%
%TCIMACRO{\U{211a} }%
%BeginExpansion
\mathbb{Q}
%EndExpansion
\right)  -GL\left(  d,%
%TCIMACRO{\U{2124} }%
%BeginExpansion
\mathbb{Z}
%EndExpansion
\right)  .$ So, $B=\left[  p_{ij}/q_{ij}\right]  _{1\leq i,j\leq d}$ where
$p_{ij},q_{ij}$ are integral values, $\gcd\left(  p_{ij},q_{ij}\right)  =1$
and $q_{ij}\neq0$ for $1\leq i,j\leq d.$ Let $m=\operatorname{lcm}\left(
q_{ij}\right)  _{1\leq i,j\leq d}.$ Clearly
\[
\chi_{l}^{m}\left(  k\right)  =1\text{ for all }l\in B%
%TCIMACRO{\U{2124} }%
%BeginExpansion
\mathbb{Z}
%EndExpansion
^{d},\text{ and }k\in%
%TCIMACRO{\U{2124} }%
%BeginExpansion
\mathbb{Z}
%EndExpansion
^{d}.
\]
Thus, $\left[  G,G\right]  $ is a finite abelian proper closed subgroup of the
circle group. As a result $\left[  G,G\right]  $ is cyclic. For part (c), if
$B\in GL\left(  d,%
%TCIMACRO{\U{211d} }%
%BeginExpansion
\mathbb{R}
%EndExpansion
\right)  -GL\left(  d,%
%TCIMACRO{\U{211a} }%
%BeginExpansion
\mathbb{Q}
%EndExpansion
\right)  $ then the commutator subgroup of $G$ is not isomorphic to a finite
subgroup of the torus. That is, there exist $l\in B%
%TCIMACRO{\U{2124} }%
%BeginExpansion
\mathbb{Z}
%EndExpansion
^{d},k\in%
%TCIMACRO{\U{2124} }%
%BeginExpansion
\mathbb{Z}
%EndExpansion
^{d}$ such that the set $\left\{  \chi_{l}^{m}\left(  k\right)  :m\in%
%TCIMACRO{\U{2124} }%
%BeginExpansion
\mathbb{Z}
%EndExpansion
\right\}  $ is not closed and is dense in the torus $\mathbb{T}$.
\end{proof}

\begin{example}
Let $G$ be group generated $T_{k},M_{l}$ such that $k\in%
%TCIMACRO{\U{2124} }%
%BeginExpansion
\mathbb{Z}
%EndExpansion
^{2},l\in B%
%TCIMACRO{\U{2124} }%
%BeginExpansion
\mathbb{Z}
%EndExpansion
^{2}$
\[
B=\left[
\begin{array}
[c]{cc}%
1/2 & 1/5\\
2/3 & -3/4
\end{array}
\right]
\]
then, $\left[  G,G\right]  $ is isomorphic to $%
%TCIMACRO{\U{2124} }%
%BeginExpansion
\mathbb{Z}
%EndExpansion
_{60}.$
\end{example}

\begin{example}
Let $G$ be group generated $T_{k},M_{l}$ such that $k\in%
%TCIMACRO{\U{2124} }%
%BeginExpansion
\mathbb{Z}
%EndExpansion
^{2},l\in B%
%TCIMACRO{\U{2124} }%
%BeginExpansion
\mathbb{Z}
%EndExpansion
^{2}$ such that
\[
B=\left[
\begin{array}
[c]{cc}%
\sqrt{2} & 1\\
-1 & 2
\end{array}
\right]
\]
then $\left[  G,G\right]  $ is isomorphic\ to an infinite subgroup of the circle.
\end{example}

Assuming that $B\ $is in $GL\left(  d,%
%TCIMACRO{\U{211a} }%
%BeginExpansion
\mathbb{Q}
%EndExpansion
\right)  ,$ we will construct an abelian normal subgroup of $G.$ For that
purpose, we will need to define the groups
\[
N_{1}=\left\langle T_{k},M_{l}\text{ }|\text{ }k\in B^{-tr}%
%TCIMACRO{\U{2124} }%
%BeginExpansion
\mathbb{Z}
%EndExpansion
^{d},l\in B%
%TCIMACRO{\U{2124} }%
%BeginExpansion
\mathbb{Z}
%EndExpansion
^{d}\right\rangle ,
\]
and
\[
N_{2}=\left\langle T_{k},M_{l}\text{ }|\text{ }k\in B^{-tr}%
%TCIMACRO{\U{2124} }%
%BeginExpansion
\mathbb{Z}
%EndExpansion
^{d}\cap%
%TCIMACRO{\U{2124} }%
%BeginExpansion
\mathbb{Z}
%EndExpansion
^{d},l\in B%
%TCIMACRO{\U{2124} }%
%BeginExpansion
\mathbb{Z}
%EndExpansion
^{d}\right\rangle .
\]
Notice that in general $N_{1}$ is not a subgroup of $G$ because the lattice $%
%TCIMACRO{\U{2124} }%
%BeginExpansion
\mathbb{Z}
%EndExpansion
^{d}$ is not invariant under the action of $B^{-tr}$ if $B^{-tr}$ has non
integral rational entries. However, the group $N_{1}$ will be important in
constructing the unitary dual of $G,$ and we will need to study some of its characteristics.

\begin{lemma}
\label{abelian}If $B$ is an element of $GL\left(  d,%
%TCIMACRO{\U{211a} }%
%BeginExpansion
\mathbb{Q}
%EndExpansion
\right)  -GL\left(  d,%
%TCIMACRO{\U{2124} }%
%BeginExpansion
\mathbb{Z}
%EndExpansion
\right)  $ then
\[
N_{1}=\left\langle T_{k},M_{l}\text{ }|\text{ }k\in B^{-tr}%
%TCIMACRO{\U{2124} }%
%BeginExpansion
\mathbb{Z}
%EndExpansion
^{d},l\in B%
%TCIMACRO{\U{2124} }%
%BeginExpansion
\mathbb{Z}
%EndExpansion
^{d}\right\rangle
\]
is an abelian group.
\end{lemma}

\begin{proof}
Given $k\in B^{-tr}%
%TCIMACRO{\U{2124} }%
%BeginExpansion
\mathbb{Z}
%EndExpansion
^{d}$ and $l\in B%
%TCIMACRO{\U{2124} }%
%BeginExpansion
\mathbb{Z}
%EndExpansion
^{d},$ we recall that $
T_{k}M_{l}T_{k}^{-1}M_{l}^{-1}f\left(  x\right)  =e^{-2\pi i\left\langle
l,k\right\rangle }f.$ Since $k\in B^{-tr}%
%TCIMACRO{\U{2124} }%
%BeginExpansion
\mathbb{Z}
%EndExpansion
^{d},$ and $l\in B%
%TCIMACRO{\U{2124} }%
%BeginExpansion
\mathbb{Z}
%EndExpansion
^{d}$ then there exist $k^{\prime},l^{\prime}\in%
%TCIMACRO{\U{2124} }%
%BeginExpansion
\mathbb{Z}
%EndExpansion
^{d}$ such that
\[
T_{k}M_{l}T_{k}^{-1}M_{l}^{-1}f=e^{-2\pi i\left\langle Bl^{\prime}%
,B^{-tr}k^{\prime}\right\rangle }f=e^{-2\pi i\left\langle l^{\prime}%
,k^{\prime}\right\rangle }f\left(  x\right)  =f.
\]
Thus, for any given $k\in B^{-tr}%
%TCIMACRO{\U{2124} }%
%BeginExpansion
\mathbb{Z}
%EndExpansion
^{d},$ and $l\in B%
%TCIMACRO{\U{2124} }%
%BeginExpansion
\mathbb{Z}
%EndExpansion
^{d},$ $T_{k}M_{l}T_{k}^{-1}M_{l}^{-1}$ is equal to the identity operator. It
follows that the commutator subgroup of $N_{1}$ is trivial.
\end{proof}

We recall the following lemmas from \cite{Unified}.

\begin{lemma}
Given two lattices $\Gamma_{1},\Gamma_{2},$ $\Gamma_{1}\cap\Gamma_{2}$ is a
lattice in $%
%TCIMACRO{\U{211d} }%
%BeginExpansion
\mathbb{R}
%EndExpansion
^{d}$ if and only if there exists a lattice $\Gamma$ that contains both
$\Gamma_{1}$ and $\Gamma_{2}.$
\end{lemma}

\begin{definition}
Let $\Gamma$ be a full-rank lattice in $%
%TCIMACRO{\U{211d} }%
%BeginExpansion
\mathbb{R}
%EndExpansion
^{d}$ with generators
\[
\left[
\begin{array}
[c]{c}%
v_{1}^{1}\\
\vdots\\
v_{1}^{d}%
\end{array}
\right]  ,\left[
\begin{array}
[c]{c}%
v_{2}^{1}\\
\vdots\\
v_{2}^{d}%
\end{array}
\right]  ,\cdots,\left[
\begin{array}
[c]{c}%
v_{d}^{1}\\
\vdots\\
v_{d}^{d}%
\end{array}
\right]  \in%
%TCIMACRO{\U{211d} }%
%BeginExpansion
\mathbb{R}
%EndExpansion
^{d}.
\]
The matrix of order $d$ below
\[
\left[
\begin{array}
[c]{ccc}%
v_{1}^{1} & \cdots & v_{d}^{1}\\
\vdots & \vdots & \vdots\\
v_{1}^{d} & \cdots & v_{d}^{d}%
\end{array}
\right]
\]
is called a basis for $\Gamma.$
\end{definition}

\begin{lemma}
\label{base} Let $\Gamma_{1},\Gamma_{2}$ be two distinct lattices with bases
$J,K$ respectively$.$ $\Gamma_{1}+\Gamma_{2},$ $\Gamma_{1}\cap\Gamma_{2}$ are
two lattices in $%
%TCIMACRO{\U{211d} }%
%BeginExpansion
\mathbb{R}
%EndExpansion
^{d}$ if and only if $JK^{-1}$is a rational matrix. \ Moreover
\[
\dim\left(  \Gamma_{1}+\Gamma_{2}\right)  +\dim\left(  \Gamma_{1}\cap
\Gamma_{2}\right)  =\dim\Gamma_{1}+\dim\Gamma_{1}%
\]

\end{lemma}

\begin{remark}
\label{proc}If $\Gamma_{1}\cap\Gamma_{2}$ is a lattice, there is a well-known
technique given in \cite{Unified} (see page $809$) used to compute the basis of the
lattice $\Gamma_{1}\cap\Gamma_{2}$. We describe the procedure here. Let $J,K$
be bases for lattices $\Gamma_{1},$ and $\Gamma_{2}$ respectively. First we
compute the $d\times2d$ matrix $\left[  J\text{ }|K\right]  .$ Secondly, we
evaluate the \textbf{Hermite lower triangular form} of $\left[  J|K\right]  .$
This form has the structure $\left[  L|0\right]  ,$ and is obtained as
$\left[  J\text{ }|K\right]  E=\left[  L|0\right]  ,$ where $E$ is matrix of
order $2d$ obtained by applying elementary row operations to $\left[  J\text{
}|K\right]  .$ In fact, $E$ is a block matrix of the type
\[
E=\left[
\begin{array}
[c]{cc}%
R & S\\
C & D
\end{array}
\right]  ,\label{procedure}%
\]
and $R,S,C,$ and $D$ are matrices of order $d.$ Finally, a basis for the
lattice $\Gamma_{1}\cap\Gamma_{2}$ is then given by $KD.$ That is $\Gamma
_{1}\cap\Gamma_{2}=\left(  KD\right)
%TCIMACRO{\U{2124} }%
%BeginExpansion
\mathbb{Z}
%EndExpansion
^{d}.$
\end{remark}

\begin{corollary}
If $B\ $is in $GL\left(  d,%
%TCIMACRO{\U{211a} }%
%BeginExpansion
\mathbb{Q}
%EndExpansion
\right)  -GL\left(  d,%
%TCIMACRO{\U{2124} }%
%BeginExpansion
\mathbb{Z}
%EndExpansion
\right)  $ then $B^{-tr}%
%TCIMACRO{\U{2124} }%
%BeginExpansion
\mathbb{Z}
%EndExpansion
^{d}\cap%
%TCIMACRO{\U{2124} }%
%BeginExpansion
\mathbb{Z}
%EndExpansion
^{d}$ is a full-rank lattice subgroup of $%
%TCIMACRO{\U{211d} }%
%BeginExpansion
\mathbb{R}
%EndExpansion
^{d}.$
\end{corollary}

\begin{proof}
The fact that $B^{-tr}%
%TCIMACRO{\U{2124} }%
%BeginExpansion
\mathbb{Z}
%EndExpansion
^{d}\cap%
%TCIMACRO{\U{2124} }%
%BeginExpansion
\mathbb{Z}
%EndExpansion
^{d}$ is a full-rank lattice follows directly from Lemma \ref{base}.
\end{proof}

We assume that $B\ $is an element of $GL\left(  d,%
%TCIMACRO{\U{211a} }%
%BeginExpansion
\mathbb{Q}
%EndExpansion
\right)  -GL\left(  d,%
%TCIMACRO{\U{2124} }%
%BeginExpansion
\mathbb{Z}
%EndExpansion
\right)  .$ We will prove that $N_{2}$ is a normal subgroup of $G$. However it
is not a maximal normal subgroup of $G$ since it does not contain the center
of the group. Thus, $N_{2}$ needs to be extended. For that purpose, we define
the group
\[
N=\left\langle T_{k},M_{l},\tau\text{ }|\text{ }k\in B^{-tr}%
%TCIMACRO{\U{2124} }%
%BeginExpansion
\mathbb{Z}
%EndExpansion
^{d}\cap%
%TCIMACRO{\U{2124} }%
%BeginExpansion
\mathbb{Z}
%EndExpansion
^{d},l\in B%
%TCIMACRO{\U{2124} }%
%BeginExpansion
\mathbb{Z}
%EndExpansion
^{d},\tau\in\left[  G,G\right]  \right\rangle \subset G.
\]

\begin{proposition}
\label{normal}If $B\ $is in $GL\left(  d,%
%TCIMACRO{\U{211a} }%
%BeginExpansion
\mathbb{Q}
%EndExpansion
\right)  $ then $N$ is an abelian normal subgroup of $G.$
\end{proposition}

\begin{proof}
From Lemma \ref{abelian}, we have already seen that $T_{k}$ commutes with
$M_{l}$ for arbitrary $k\in B^{-tr}%
%TCIMACRO{\U{2124} }%
%BeginExpansion
\mathbb{Z}
%EndExpansion
^{d}\cap%
%TCIMACRO{\U{2124} }%
%BeginExpansion
\mathbb{Z}
%EndExpansion
^{d},l\in B%
%TCIMACRO{\U{2124} }%
%BeginExpansion
\mathbb{Z}
%EndExpansion
^{d}.$ Since $\left[  G,G\right]  $ commutes with $T_{k}$ and $M_{l}$ for
$k\in B^{-tr}%
%TCIMACRO{\U{2124} }%
%BeginExpansion
\mathbb{Z}
%EndExpansion
^{d}\cap%
%TCIMACRO{\U{2124} }%
%BeginExpansion
\mathbb{Z}
%EndExpansion
^{d},l\in B%
%TCIMACRO{\U{2124} }%
%BeginExpansion
\mathbb{Z}
%EndExpansion
^{d}$, then $N$ is abelian. For the second part of the proof, let $k\in
B^{-tr}%
%TCIMACRO{\U{2124} }%
%BeginExpansion
\mathbb{Z}
%EndExpansion
^{d}\cap%
%TCIMACRO{\U{2124} }%
%BeginExpansion
\mathbb{Z}
%EndExpansion
^{d},l\in B%
%TCIMACRO{\U{2124} }%
%BeginExpansion
\mathbb{Z}
%EndExpansion
^{d}$ and $s\in%
%TCIMACRO{\U{2124} }%
%BeginExpansion
\mathbb{Z}
%EndExpansion
^{d}.$ First, we compute the conjugation action of the translation operator on
an arbitrary element of $N.$ Let $s\in%
%TCIMACRO{\U{2124} }%
%BeginExpansion
\mathbb{Z}
%EndExpansion
^{d}.$ Then $
T_{s}\left(  T_{k}M_{l}\tau\right)  T_{s}^{-1}=\tau e^{-2\pi i\left\langle
l,s\right\rangle }T_{k}M_{l}.$ Next, we compute the conjugation action of the modulation operator on an
arbitrary element of $N$ as follows: $
M_{s}\left(  T_{k}M_{l}\tau\right)  M_{s}^{-1}=\tau T_{k}M_{l}.$ Clearly, $GNG^{-1}\subset N.$ Thus, $N$ is a normal subgroup of $G.$
\end{proof}

\begin{lemma}\label{finite}
Assuming that $B\ $is in $GL\left(  d,%
%TCIMACRO{\U{211a} }%
%BeginExpansion
\mathbb{Q}
%EndExpansion
\right)  -GL\left(  d,%
%TCIMACRO{\U{2124} }%
%BeginExpansion
\mathbb{Z}
%EndExpansion
\right)  ,$ then the following holds.
\begin{enumerate}
\item The quotient group $
\frac{%
%TCIMACRO{\U{2124} }%
%BeginExpansion
\mathbb{Z}
%EndExpansion
^{d}}{B^{-tr}%
%TCIMACRO{\U{2124} }%
%BeginExpansion
\mathbb{Z}
%EndExpansion
^{d}\cap%
%TCIMACRO{\U{2124} }%
%BeginExpansion
\mathbb{Z}
%EndExpansion
^{d}}$ is isomorphic to a finite abelian group.
\item The group $G/N$ is a finite group isomorphic to  $
\frac{%
%TCIMACRO{\U{2124} }%
%BeginExpansion
\mathbb{Z}
%EndExpansion
^{d}}{B^{-tr}%
%TCIMACRO{\U{2124} }%
%BeginExpansion
\mathbb{Z}
%EndExpansion
^{d}\cap%
%TCIMACRO{\U{2124} }%
%BeginExpansion
\mathbb{Z}
%EndExpansion
^{d}}.$
\end{enumerate}
\end{lemma}

\begin{proof}
For the first part of the proof, since $B^{-tr}%
%TCIMACRO{\U{2124} }%
%BeginExpansion
\mathbb{Z}
%EndExpansion
^{d}\cap%
%TCIMACRO{\U{2124} }%
%BeginExpansion
\mathbb{Z}
%EndExpansion
^{d}$ is a full-rank lattice, there exists a non-singular matrix $A$ such that
$B^{-tr}%
%TCIMACRO{\U{2124} }%
%BeginExpansion
\mathbb{Z}
%EndExpansion
^{d}\cap%
%TCIMACRO{\U{2124} }%
%BeginExpansion
\mathbb{Z}
%EndExpansion
^{d}=A%
%TCIMACRO{\U{2124} }%
%BeginExpansion
\mathbb{Z}
%EndExpansion
^{d}.$ Thus, referring to the discussion \cite{Unified} (page $95)$, the index
of $B^{-tr}%
%TCIMACRO{\U{2124} }%
%BeginExpansion
\mathbb{Z}
%EndExpansion
^{d}\cap%
%TCIMACRO{\U{2124} }%
%BeginExpansion
\mathbb{Z}
%EndExpansion
^{d}$ in $%
%TCIMACRO{\U{2124} }%
%BeginExpansion
\mathbb{Z}
%EndExpansion
^{d}$ is $\left(
%TCIMACRO{\U{2124} }%
%BeginExpansion
\mathbb{Z}
%EndExpansion
^{d}:B^{-tr}%
%TCIMACRO{\U{2124} }%
%BeginExpansion
\mathbb{Z}
%EndExpansion
^{d}\cap%
%TCIMACRO{\U{2124} }%
%BeginExpansion
\mathbb{Z}
%EndExpansion
^{d}\right)  =\left\vert \det A\right\vert .$ As a result,
\[
\frac{%
%TCIMACRO{\U{2124} }%
%BeginExpansion
\mathbb{Z}
%EndExpansion
^{d}}{B^{-tr}%
%TCIMACRO{\U{2124} }%
%BeginExpansion
\mathbb{Z}
%EndExpansion
^{d}\cap%
%TCIMACRO{\U{2124} }%
%BeginExpansion
\mathbb{Z}
%EndExpansion
^{d}}%
\]
is a finite abelian group.  For the second part of lemma,  there exist $k_{1},k_{2},\cdots,k_{m}\in%
%TCIMACRO{\U{2124} }%
%BeginExpansion
\mathbb{Z}
%EndExpansion
^{d}$ such that
\begin{align*}
\frac{%
%TCIMACRO{\U{2124} }%
%BeginExpansion
\mathbb{Z}
%EndExpansion
^{d}}{B^{-tr}%
%TCIMACRO{\U{2124} }%
%BeginExpansion
\mathbb{Z}
%EndExpansion
^{d}\cap%
%TCIMACRO{\U{2124} }%
%BeginExpansion
\mathbb{Z}
%EndExpansion
^{d}}  &  =\left\{  \left(  B^{-tr}%
%TCIMACRO{\U{2124} }%
%BeginExpansion
\mathbb{Z}
%EndExpansion
^{d}\cap%
%TCIMACRO{\U{2124} }%
%BeginExpansion
\mathbb{Z}
%EndExpansion
^{d}\right)  ,k_{1}+\left(  B^{-tr}%
%TCIMACRO{\U{2124} }%
%BeginExpansion
\mathbb{Z}
%EndExpansion
^{d}\cap%
%TCIMACRO{\U{2124} }%
%BeginExpansion
\mathbb{Z}
%EndExpansion
^{d}\right)  ,\cdots,k_{m}+\left(  B^{-tr}%
%TCIMACRO{\U{2124} }%
%BeginExpansion
\mathbb{Z}
%EndExpansion
^{d}\cap%
%TCIMACRO{\U{2124} }%
%BeginExpansion
\mathbb{Z}
%EndExpansion
^{d}\right)  \right\} \\
&  \cong\left\{  1,T_{k_{1}}\cdots,T_{k_{m}}\right\}.
\end{align*}
Also,
\begin{align*}
G/N  &  \cong\left\{  1,T_{k_{1}}\cdots,T_{k_{m}}\right\} \\
&  \cong\frac{%
%TCIMACRO{\U{2124} }%
%BeginExpansion
\mathbb{Z}
%EndExpansion
^{d}}{B^{-tr}%
%TCIMACRO{\U{2124} }%
%BeginExpansion
\mathbb{Z}
%EndExpansion
^{d}\cap%
%TCIMACRO{\U{2124} }%
%BeginExpansion
\mathbb{Z}
%EndExpansion
^{d}}%
\end{align*}
and this completes the proof.
\end{proof}

\section{The Plancherel Measure and Application: The Integer Case}

If $B\in GL\left(  d,%
%TCIMACRO{\U{2124} }%
%BeginExpansion
\mathbb{Z}
%EndExpansion
\right)  $ then $G$ is commutative and isomorphic to $%
%TCIMACRO{\U{2124} }%
%BeginExpansion
\mathbb{Z}
%EndExpansion
^{d}\times B%
%TCIMACRO{\U{2124} }%
%BeginExpansion
\mathbb{Z}
%EndExpansion
^{d}$. In this case, the unitary dual, and the Plancherel measure of $
G=\left\langle T_{k},M_{l}:k\in%
%TCIMACRO{\U{2124} }%
%BeginExpansion
\mathbb{Z}
%EndExpansion
^{d},l\in B%
%TCIMACRO{\U{2124} }%
%BeginExpansion
\mathbb{Z}
%EndExpansion
^{d}\right\rangle$ are well-understood through the Pontrjagin duality. Let $L$ be the left
regular representation of $G.$ The unitary dual of $G$ is isomorphic to
\[
\frac{\widehat{\mathbb{%
%TCIMACRO{\U{211d} }%
%BeginExpansion
\mathbb{R}
%EndExpansion
}^{d}}}{%
%TCIMACRO{\U{2124} }%
%BeginExpansion
\mathbb{Z}
%EndExpansion
^{d}}\times\frac{\widehat{\mathbb{%
%TCIMACRO{\U{211d} }%
%BeginExpansion
\mathbb{R}
%EndExpansion
}^{d}}}{B^{-tr}%
%TCIMACRO{\U{2124} }%
%BeginExpansion
\mathbb{Z}
%EndExpansion
^{d}}%
\]
and the Plancherel measure is up to multiplication by a constant equal to
\[
\frac{dxdy}{\left\vert \det\left(  B^{-tr}\right)  \right\vert }%
\]
which is supported on a measurable set $\Lambda\subset%
%TCIMACRO{\U{211d} }%
%BeginExpansion
\mathbb{R}
%EndExpansion
^{2d}$ parametrizing the group $\frac{\widehat{\mathbb{%
%TCIMACRO{\U{211d} }%
%BeginExpansion
\mathbb{R}
%EndExpansion
}^{d}}}{%
%TCIMACRO{\U{2124} }%
%BeginExpansion
\mathbb{Z}
%EndExpansion
^{d}}\times\frac{\widehat{\mathbb{%
%TCIMACRO{\U{211d} }%
%BeginExpansion
\mathbb{R}
%EndExpansion
}^{d}}}{B^{-tr}%
%TCIMACRO{\U{2124} }%
%BeginExpansion
\mathbb{Z}
%EndExpansion
^{d}}.$ More precisely, the collection of sets
\[
\left\{  \Lambda+j:j\in%
%TCIMACRO{\U{2124} }%
%BeginExpansion
\mathbb{Z}
%EndExpansion
^{d}\times B^{-tr}%
%TCIMACRO{\U{2124} }%
%BeginExpansion
\mathbb{Z}
%EndExpansion
^{d}\right\}
\]
forms a measurable partition for $\mathbb{%
%TCIMACRO{\U{211d} }%
%BeginExpansion
\mathbb{R}
%EndExpansion
}^{2d}$ and $\Lambda$ is called the \textbf{spectrum} of the left regular
representation of $\pi.$ It is worth noticing that there is no canonical way
to choose $\Lambda.$ Moreover, via the Plancherel transform, the left regular
representation of $%
%TCIMACRO{\U{2124} }%
%BeginExpansion
\mathbb{Z}
%EndExpansion
^{d}\times B%
%TCIMACRO{\U{2124} }%
%BeginExpansion
\mathbb{Z}
%EndExpansion
^{d}$ is decomposed into a direct integral decomposition of characters
$\int_{\Lambda}^{\oplus}\chi_{\left(  x,y\right)  }\left\vert \det
B\right\vert dxdy$ acting in $\int_{\Lambda}^{\oplus}%
%TCIMACRO{\U{2102} }%
%BeginExpansion
\mathbb{C}
%EndExpansion
$ $\left\vert \det B\right\vert dxdy\cong L^{2}\left(  \Lambda\right)  .$

Now, we will discuss some application of the Plancherel theory of $G$ to Gabor
theory. Let $B\in GL\left(  d,%
%TCIMACRO{\U{2124} }%
%BeginExpansion
\mathbb{Z}
%EndExpansion
\right)  ,$ and let $\widetilde{G}=%
%TCIMACRO{\U{2124} }%
%BeginExpansion
\mathbb{Z}
%EndExpansion
^{d}\times B%
%TCIMACRO{\U{2124} }%
%BeginExpansion
\mathbb{Z}
%EndExpansion
^{d}$. There is a representation of $\widetilde{G}$ where $\pi:\widetilde
{G}\rightarrow U\left(  L^{2}\left(
%TCIMACRO{\U{211d} }%
%BeginExpansion
\mathbb{R}
%EndExpansion
^{d}\right)  \right)  ,$ such that $G$ is the image of $\widetilde{G}$ via the
representation $\pi.$ The representation $\pi$ is called a \textbf{Gabor
representation} of $\widetilde{G}.$ Since $G$ is abelian, once a choice for a
transversal of $\frac{\widehat{\mathbb{%
%TCIMACRO{\U{211d} }%
%BeginExpansion
\mathbb{R}
%EndExpansion
}^{d}}}{%
%TCIMACRO{\U{2124} }%
%BeginExpansion
\mathbb{Z}
%EndExpansion
^{d}}\times\frac{\widehat{\mathbb{%
%TCIMACRO{\U{211d} }%
%BeginExpansion
\mathbb{R}
%EndExpansion
}^{d}}}{B^{-tr}%
%TCIMACRO{\U{2124} }%
%BeginExpansion
\mathbb{Z}
%EndExpansion
^{d}}$ is made, there is a decomposition of $\pi$ into irreducible
representations of the group $\widetilde{G}$
\begin{equation}
\int_{\Lambda}^{\oplus}\chi_{\varsigma}\otimes\mathbf{1}_{%
%TCIMACRO{\U{2102} }%
%BeginExpansion
\mathbb{C}
%EndExpansion
^{n\left(  \varsigma\right)  }}d\mu\left(  \varsigma\right)  \label{decomp}%
\end{equation}
acting in $\int_{\Lambda}^{\oplus}%
%TCIMACRO{\U{2102} }%
%BeginExpansion
\mathbb{C}
%EndExpansion
\otimes%
%TCIMACRO{\U{2102} }%
%BeginExpansion
\mathbb{C}
%EndExpansion
^{n\left(  \varsigma\right)  }d\mu\left(  \varsigma\right)  .$ The function
$n:\Lambda\rightarrow%
%TCIMACRO{\U{2115} }%
%BeginExpansion
\mathbb{N}
%EndExpansion
\cup\left\{  0\right\}  $ is the \textbf{multiplicity function} of the
irreducible representations appearing in the decomposition of $\pi.$

Let us define $W_{f}:L^{2}\left(
%TCIMACRO{\U{211d} }%
%BeginExpansion
\mathbb{R}
%EndExpansion
^{d}\right)  \rightarrow l^{2}\left(  \widetilde{G}\right)  $ where
\[
W_{f}h\left(  k,l\right)  =\left\langle h,T_{k}M_{l}f\right\rangle .
\]
We recall that $\pi$ is \textbf{admissible} if and only if $W_{f}$ defines an isometry on
$L^{2}\left(
%TCIMACRO{\U{211d} }%
%BeginExpansion
\mathbb{R}
%EndExpansion
^{d}\right).$

\begin{lemma}
Let $B\in GL\left(  d,%
%TCIMACRO{\U{2124} }%
%BeginExpansion
\mathbb{Z}
%EndExpansion
\right)  .$ $\pi$ is admissible if and only if
\[%
%TCIMACRO{\dsum \limits_{T_{k}M_{l}\in G}}%
%BeginExpansion
{\displaystyle\sum\limits_{T_{k}M_{l}\in G}}
%EndExpansion
\left\vert \left\langle h,T_{k}M_{l}f\right\rangle \right\vert ^{2}=\left\Vert
h\right\Vert ^{2}
\]
for all $h\in L^2\left(\mathbb{R}^d\right).$ That is $f$ is a \textbf{Gabor Parseval frame}.
\end{lemma}

We recall that the \textbf{Zak transform} is a unitary operator
\[
Z:L^{2}\left(
%TCIMACRO{\U{211d} }%
%BeginExpansion
\mathbb{R}
%EndExpansion
^{d}\right)  \rightarrow L^{2}\left(  \left[  0,1\right)  ^{d}\times\left[
0,1\right)  ^{d}\right)  \cong\int_{\left[  0,1\right)  ^{d}\times\left[
0,1\right)  ^{d}}^{\oplus}%
%TCIMACRO{\U{2102} }%
%BeginExpansion
\mathbb{C}
%EndExpansion
dxdy
\]
where
\[
Zf\left(  x,y\right)  =%
%TCIMACRO{\dsum \limits_{m\in\mathbb{Z}^{d}}}%
%BeginExpansion
{\displaystyle\sum\limits_{m\in\mathbb{Z}^{d}}}
%EndExpansion
f\left(  x+m\right)  \exp\left(  2\pi i\left\langle m,y\right\rangle \right)
.
\]
It is easy to see that if $B\in GL\left(  d,%
%TCIMACRO{\U{2124} }%
%BeginExpansion
\mathbb{Z}
%EndExpansion
\right)  $ then
\[
Z\left(  T_{k}M_{l}f\right)  \left(  x,y\right)  =\exp\left(  -2\pi
i\left\langle k,y\right\rangle \right)  \exp\left(  2\pi i\left\langle
l,x\right\rangle \right)  Zf\left(  x,y\right)  .
\]
Thus, the Zak transform intertwines the Gabor representation with the
representation
\begin{equation}
\int_{\left[  0,1\right)  ^{d}\times\left[  0,1\right)  ^{d}}^{\oplus}%
\chi_{\left(  x,y\right)  }dxdy. \label{zaky}%
\end{equation}
Now, we would like to compare the representations given in (\ref{decomp}) and
(\ref{zaky}).

\begin{proposition}
\label{gr}If $B\in GL\left(  d,%
%TCIMACRO{\U{2124} }%
%BeginExpansion
\mathbb{Z}
%EndExpansion
\right)  $ and $\left\vert \det B\right\vert \neq1$ then the Gabor
representation is not admissible. That is, there is no Parseval frame of the
type $\pi\left(  \widetilde{G}\right)  f.$ Moreover, the Gabor representation
is unitarily equivalent to the direct integral
\[
\int_{\Lambda}^{\oplus}\chi_{\varsigma}\otimes\mathbf{1}_{%
%TCIMACRO{\U{2102} }%
%BeginExpansion
\mathbb{C}
%EndExpansion
^{\left\vert \det B\right\vert }}d\mu\left(  \varsigma\right)  .
\]

\end{proposition}

\begin{proof}
First, let us notice that if $B\in GL\left(  d,%
%TCIMACRO{\U{2124} }%
%BeginExpansion
\mathbb{Z}
%EndExpansion
\right)  $ then $\left\vert \det B^{-tr}\right\vert \leq1.$ As a result, there
is a measurable set $\Lambda\subset\mathbb{%
%TCIMACRO{\U{211d} }%
%BeginExpansion
\mathbb{R}
%EndExpansion
}^{2d}$ tiling $\mathbb{%
%TCIMACRO{\U{211d} }%
%BeginExpansion
\mathbb{R}
%EndExpansion
}^{2d}$ by the lattice $%
%TCIMACRO{\U{2124} }%
%BeginExpansion
\mathbb{Z}
%EndExpansion
^{d}\times B^{-tr}%
%TCIMACRO{\U{2124} }%
%BeginExpansion
\mathbb{Z}
%EndExpansion
^{d}$ such that $\Lambda$ is contained in $\left[
0,1\right)  ^{d}\times\left[  0,1\right)  ^{d}$ (see \cite{Han}). Thus if
$\left\vert \det B\right\vert \neq1,$ the representation $\int_{\left[
0,1\right)  ^{d}}^{\oplus}\int_{\left[  0,1\right)  ^{d}}^{\oplus}%
\chi_{\left(  x,y\right)  }dxdy$ cannot be contained in $\int_{\Lambda
}^{\oplus}\chi_{\left(  x,y\right)  }\left\vert \det\left(  B^{tr}\right)
\right\vert dxdy.$ Picking $\Lambda=\left[  0,1\right)  ^{d}\times
E\subset\mathbb{%
%TCIMACRO{\U{211d} }%
%BeginExpansion
\mathbb{R}
%EndExpansion
}^{2d}$ such that $E\subseteq\left[  0,1\right)  ^{d},$ we obtain
\begin{equation}
\pi\cong\int_{\Lambda}^{\oplus}\chi_{\varsigma}\otimes\mathbf{1}_{%
%TCIMACRO{\U{2102} }%
%BeginExpansion
\mathbb{C}
%EndExpansion
^{n\left(  \varsigma\right)  }}d\mu\left(  \varsigma\right)
\label{decomposition}%
\end{equation}
We now claim that the multiplicity function is given by
\[
n\left(  \varsigma\right)  =\mathrm{\#}\left(  \left\{  j\in B^{-tr}%
%TCIMACRO{\U{2124} }%
%BeginExpansion
\mathbb{Z}
%EndExpansion
^{d}:\left\{  \varsigma+j\right\}  \cap\left[  0,1\right)  ^{d}\neq
\emptyset\right\}  \right)  =\left\vert \det B\right\vert .
\]
To show that the above holds, we partition $\left[  0,1\right)  ^{d}%
\times\left[  0,1\right)  ^{d}$ into $\left\vert \det B\right\vert $ many
subsets $\Lambda^{k}$ such that each set $\Lambda^{k}$ is a fundamental domain
for $%
%TCIMACRO{\U{2124} }%
%BeginExpansion
\mathbb{Z}
%EndExpansion
^{d}\times B^{-tr}%
%TCIMACRO{\U{2124} }%
%BeginExpansion
\mathbb{Z}
%EndExpansion
^{d}.$ Writing
\[
\left[  0,1\right)  ^{d}\times\left[  0,1\right)  ^{d}=%
%TCIMACRO{\dbigcup \limits_{k=1}^{\left\vert \det B\right\vert }}%
%BeginExpansion
{\displaystyle\bigcup\limits_{k=1}^{\left\vert \det B\right\vert }}
%EndExpansion
\Lambda^{k},
\]
we obtain
\begin{align*}
\pi &  \cong\int_{\left[  0,1\right)  ^{d}\times\left[  0,1\right)  ^{d}%
}^{\oplus}\chi_{\left(  x,y\right)  }dxdy\\
&  \cong\int_{\Lambda
}^{\oplus}%
%TCIMACRO{\dbigoplus \limits_{k=1}^{\left\vert \det B\right\vert }}%
%BeginExpansion
{\displaystyle\bigoplus\limits_{k=1}^{\left\vert \det B\right\vert }}
%EndExpansion
\chi_{\left(  x,y\right)  }dxdy\\
&  \cong\int_{\Lambda}^{\oplus}\chi_{\varsigma}\otimes\mathbf{1}_{%
%TCIMACRO{\U{2102} }%
%BeginExpansion
\mathbb{C}
%EndExpansion
^{\left\vert \det B\right\vert }}d\mu\left(  \varsigma\right)  .
\end{align*}

\end{proof}

\begin{example}
Let $
G=\left\langle T_{k},M_{l}\text{ }|\text{ }k\in%
%TCIMACRO{\U{2124} }%
%BeginExpansion
\mathbb{Z}
%EndExpansion
,l\in3%
%TCIMACRO{\U{2124} }%
%BeginExpansion
\mathbb{Z}
%EndExpansion
\right\rangle .$ The spectrum of the left regular representation of $G$ is given by $
\Lambda=\left[  0,1\right)  \times\left[  0,1/3\right)$ and $\pi\cong\int_{\Lambda}^{\oplus}\chi_{\varsigma}\otimes\mathbf{1}_{%
%TCIMACRO{\U{2102} }%
%BeginExpansion
\mathbb{C}
%EndExpansion
^{3}}d\mu\left(  \varsigma\right)  $. Thus, $\pi$ is \textbf{not} an
admissible representation since $\pi\cong L\oplus L\oplus L.$
\end{example}

\begin{example}
Let $
G=\left\langle T_{k},M_{l}\text{ }|\text{ }k\in%
%TCIMACRO{\U{2124} }%
%BeginExpansion
\mathbb{Z}
%EndExpansion
^{2},l\in B%
%TCIMACRO{\U{2124} }%
%BeginExpansion
\mathbb{Z}
%EndExpansion
^{2}\right\rangle$ where
\[
B=\left[
\begin{array}
[c]{cc}%
3 & 1\\
1 & 2
\end{array}
\right]  .
\]
The spectrum of the left regular representation is%
\[
\left[  0,1\right)  ^{2}\times\left[
\begin{array}
[c]{cc}%
\frac{2}{5} & -\frac{1}{5}\\
-\frac{1}{5} & \frac{3}{5}%
\end{array}
\right]  \left[  0,1\right)  ^{2}.
\]
In the graph below we illustrate the idea that there exists a collection of
sets $
\left\{  \Lambda^{k}\right\}  _{k=1}^{5}$ such that each set $\Lambda^{k}$ is a fundamental domain of $\left[
0,1\right)  ^{2}\times B^{-tr}\left[  0,1\right)  ^{2}$ and the spectrum of
$\pi$ is given as follows
\[
\left[  0,1\right)  ^{2}\times\left[  0,1\right)  ^{2}=%
%TCIMACRO{\dbigcup \limits_{k=1}^{5}}%
%BeginExpansion
{\displaystyle\bigcup\limits_{k=1}^{5}}
%EndExpansion
\Lambda^{k}.
\]
Thus, $\pi\cong L\oplus L\oplus L\oplus L\oplus L.$ As a result, the Gabor
representation $\pi$ is not admissible.
\[%
%TCIMACRO{\FRAME{itbpFU}{3in}{2.9992in}{0in}{\Qcb{A projection of the spectrum
%of $\pi$ in $\U{211d} ^2.$}}{}{multiplicity.eps}%
%{\special{ language "Scientific Word";  type "GRAPHIC";  display "USEDEF";
%valid_file "F";  width 3in;  height 2.9992in;  depth 0in;
%original-width 0pt;  original-height 0pt;  cropleft "0";  croptop "1";
%cropright "1";  cropbottom "0";
%filename 'multiplicity.eps';file-properties "XNPEU";}} }%
%BeginExpansion
{\parbox[b]{3in}
{\begin{center}
\includegraphics[
height=2.9992in,
width=3in
]%
{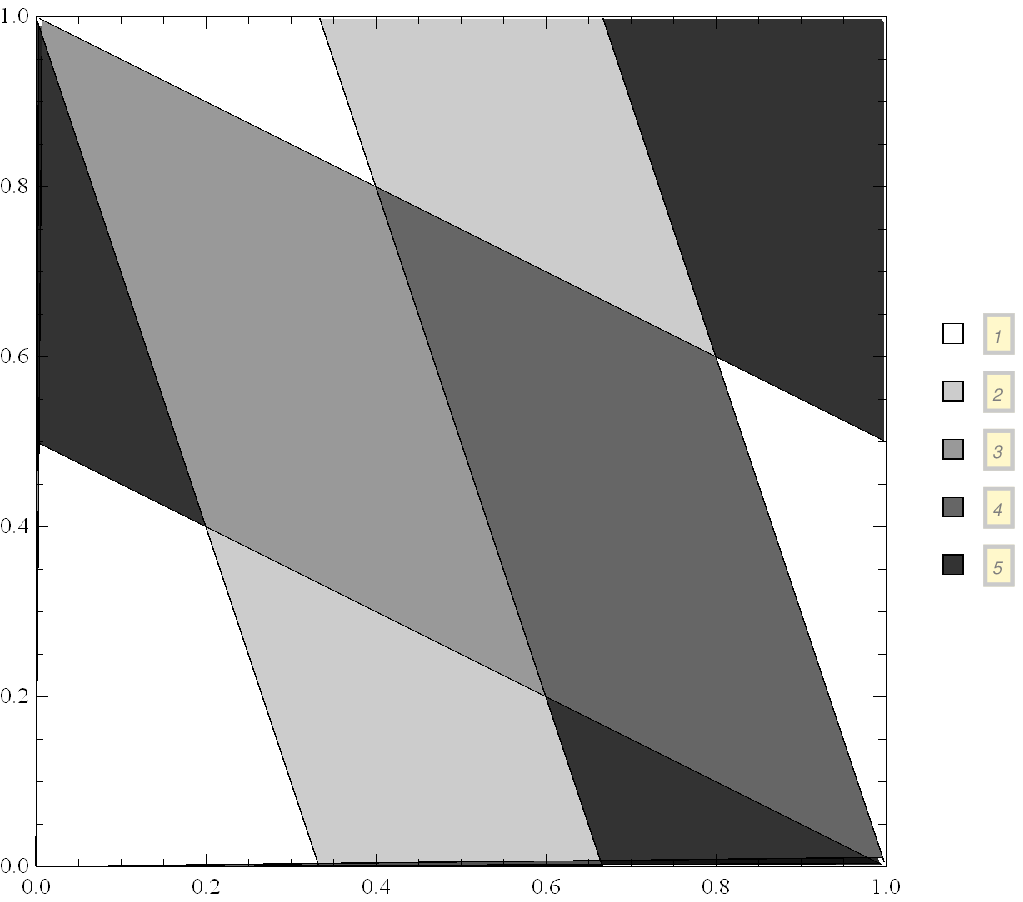}%
\\
A projection of the spectrum of $\pi$ in $\mathbb{R}^2.$
\end{center}}}
%EndExpansion
\]

\end{example}

\begin{remark}
Let $\mathcal{K}$ be a $\pi$-invariant closed subspace of $L^{2}\left(
%TCIMACRO{\U{211d} }%
%BeginExpansion
\mathbb{R}
%EndExpansion
^{d}\right)  $ ($\mathcal{K}$ is a shift-modulation invariant space
\cite{Bownik}). There exists a unitary operator
\[
J:\mathcal{K}\rightarrow\int_{\left[  0,1\right)  ^{d}\times B^{-tr}\left[
0,1\right)  ^{d}}^{\oplus}%
%TCIMACRO{\U{2102} }%
%BeginExpansion
\mathbb{C}
%EndExpansion
\otimes%
%TCIMACRO{\U{2102} }%
%BeginExpansion
\mathbb{C}
%EndExpansion
^{\overline{n}\left(  \varsigma\right)  }d\mu\left(  \varsigma\right)
\]
intertwining the representations $\pi|\mathcal{K}$ with $
\int_{\Lambda}^{\oplus}\chi_{\varsigma}\otimes\mathbf{1}_{%
%TCIMACRO{\U{2102} }%
%BeginExpansion
\mathbb{C}
%EndExpansion
^{\overline{n}\left(  \varsigma\right)  }}d\mu\left(  \varsigma\right)$ and $\overline{n}\left(  \varsigma\right)  \leq\left\vert \det B\right\vert $
a.e. As a result, we have a general characterization of shift-modulation
invariant spaces in the specific case where $B$ has integral entries.

\begin{proof}
The proof follows from the fact that $\left(  \pi|\mathcal{K}\text{,}%
\mathcal{K}\right)  $ is a subrepresentation of the Gabor representation of
$\widetilde{G}.$
\end{proof}
\end{remark}

\begin{definition}
Two unitary representations $\left(  \pi_{1},H_{1}\right)  ,\left(  \pi
_{2},H_{2}\right)  $ of a group $X$ are quasi-equivalent if there exist
unitarily-equivalent representations $\rho_{1},\rho_{2}$ such that $\rho_{k}$
is a multiple of $\pi_{k}$ for $k=1,2.$
\end{definition}

\begin{remark}\label{rem}
If $B\in GL\left(  d,%
%TCIMACRO{\U{2124} }%
%BeginExpansion
\mathbb{Z}
%EndExpansion
\right)  $ then the Gabor representation $\pi$ is quasi-equivalent to the left
regular representation of $\widetilde{G}.$
\end{remark}

The proof of Remark \ref{rem} is a direct application of Proposition \ref{gr}.

\begin{remark}
In the case where $B$ is the identity matrix, $\pi$ is an admissible
representation. In fact, a well-known admissible vector is the indicator
function of the cube $\left[  0,1\right)  ^{d}.$
\end{remark}

\section{The Plancherel Measure and Application: The Rational Case}

In this section, we assume that the given matrix $B$ has at least one rational
non-integral entry. We recall from Lemma \ref{GG} that
\[
G=\left\langle T_{k},M_{l}:k\in%
%TCIMACRO{\U{2124} }%
%BeginExpansion
\mathbb{Z}
%EndExpansion
^{d},l\in B%
%TCIMACRO{\U{2124} }%
%BeginExpansion
\mathbb{Z}
%EndExpansion
^{d}\right\rangle
\]
is not commutative but is a discrete type I group. Thus, its unitary dual
exists, and its Plancherel measure is computable. Using \textbf{Mackey's
Machine} and results developed by Ronald Lispman and Adam Kleppner in
\cite{Lipsman1}, we will describe the unitary dual of the group $G$, and a
formula for the Plancherel measure. We recall that if $B\ $is in $GL\left(  d,%
%TCIMACRO{\U{211a} }%
%BeginExpansion
\mathbb{Q}
%EndExpansion
\right)  -GL\left(  d,%
%TCIMACRO{\U{2124} }%
%BeginExpansion
\mathbb{Z}
%EndExpansion
\right)  $ then $G$ contains a normal subgroup
\[
N=\left\langle T_{k},M_{l},\tau\text{ }|\text{ }k\in B^{-tr}%
%TCIMACRO{\U{2124} }%
%BeginExpansion
\mathbb{Z}
%EndExpansion
^{d}\cap%
%TCIMACRO{\U{2124} }%
%BeginExpansion
\mathbb{Z}
%EndExpansion
^{d},l\in B%
%TCIMACRO{\U{2124} }%
%BeginExpansion
\mathbb{Z}
%EndExpansion
^{d},\tau\in\left[  G,G\right]  \right\rangle
\]
which is isomorphic with a direct product of abelian groups. Since $N$ is
isomorphic to
\[
\left(  B^{-tr}%
%TCIMACRO{\U{2124} }%
%BeginExpansion
\mathbb{Z}
%EndExpansion
^{d}\cap%
%TCIMACRO{\U{2124} }%
%BeginExpansion
\mathbb{Z}
%EndExpansion
^{d}\right)  \times B%
%TCIMACRO{\U{2124} }%
%BeginExpansion
\mathbb{Z}
%EndExpansion
^{d}\times%
%TCIMACRO{\U{2124} }%
%BeginExpansion
\mathbb{Z}
%EndExpansion
_{m},
\]
its unitary dual is a group of characters. The underlying set for the group
$G$ is $%
%TCIMACRO{\U{2124} }%
%BeginExpansion
\mathbb{Z}
%EndExpansion
^{d}\times B%
%TCIMACRO{\U{2124} }%
%BeginExpansion
\mathbb{Z}
%EndExpansion
^{d}\times%
%TCIMACRO{\U{2124} }%
%BeginExpansion
\mathbb{Z}
%EndExpansion
_{m}$, and we define the representation $\pi$ of the group $G$ as follows:
\begin{equation}
\pi\left(  k,l,j\right)  =T_{k}M_{l}e^{\frac{2\pi ij}{m}} \label{repPi}%
\end{equation}

\begin{lemma}
Let $B\in GL\left(  d,%
%TCIMACRO{\U{211a} }%
%BeginExpansion
\mathbb{Q}
%EndExpansion
\right)  .$ If $\pi$ is admissible then there exists a function $\phi\in L^{2}\left(
%TCIMACRO{\U{211d} }%
%BeginExpansion
\mathbb{R}
%EndExpansion
^{d}\right)  $ such that given $h\in L^{2}\left(
%TCIMACRO{\U{211d} }%
%BeginExpansion
\mathbb{R}
%EndExpansion
^{d}\right)  ,$
\[%
%TCIMACRO{\dsum \limits_{T_{k}M_{l}\in G}}%
%BeginExpansion
{\displaystyle\sum\limits_{T_{k}M_{l}\in G}}
%EndExpansion
\left\vert \left\langle h,T_{k}M_{l}\phi  \right\rangle \right\vert ^{2}=\left\Vert
h\right\Vert ^{2}.
\]
That is $\pi\left(  G\right) \phi$
is \textbf{Parseval Gabor frame}.
\end{lemma}

\begin{proof}
Let $h\in L^{2}\left(
%TCIMACRO{\U{211d} }%
%BeginExpansion
\mathbb{R}
%EndExpansion
^{d}\right)  .$ If $\pi$ is admissible, and if $f$ is an admissible vector,
\begin{align*}%
%TCIMACRO{\dsum \limits_{T_{k}M_{l}e^{2\pi i\theta_{2}}\in G}}%
%BeginExpansion
{\displaystyle\sum\limits_{T_{k}M_{l}e^{2\pi i\theta_{2}}\in G}}
%EndExpansion
\left\vert \left\langle h,T_{k}M_{l}e^{2\pi i\theta_{2}}f\right\rangle
\right\vert ^{2}  &  =%
%TCIMACRO{\dsum \limits_{T_{k}M_{l}}}%
%BeginExpansion
{\displaystyle\sum\limits_{T_{k}M_{l}}}
%EndExpansion%
%TCIMACRO{\dsum \limits_{e^{2\pi i\theta_{2}}\in\left[  G,G\right]  }}%
%BeginExpansion
{\displaystyle\sum\limits_{e^{2\pi i\theta_{2}}\in\left[  G,G\right]  }}
%EndExpansion
\left\vert \left\langle h,T_{k}M_{l}f\right\rangle \right\vert ^{2}\\
&  =%
%TCIMACRO{\dsum \limits_{T_{k}M_{l}}}%
%BeginExpansion
{\displaystyle\sum\limits_{T_{k}M_{l}}}
%EndExpansion
\#\left(  \left[  G,G\right]  \right)  \left\vert \left\langle h,T_{k}%
M_{l}f\right\rangle \right\vert ^{2}\\
&  =%
%TCIMACRO{\dsum \limits_{T_{k}M_{l}}}%
%BeginExpansion
{\displaystyle\sum\limits_{T_{k}M_{l}}}
%EndExpansion
\left\vert \left\langle h,T_{k}M_{l}\#\left(  \left[  G,G\right]  \right)
^{1/2}f\right\rangle \right\vert ^{2}\\
&  =\left\Vert h\right\Vert ^{2}.
\end{align*}
Thus, the statement of the lemma holds by replacing $\#\left(  \left[  G,G\right]  \right)
^{1/2}f$ with $\phi.$
\end{proof}

Using the procedure provided in Remark \ref{proc}, we construct an invertible
matrix $A$ with integral entries such that $B^{-tr}%
%TCIMACRO{\U{2124} }%
%BeginExpansion
\mathbb{Z}
%EndExpansion
^{d}\cap%
%TCIMACRO{\U{2124} }%
%BeginExpansion
\mathbb{Z}
%EndExpansion
^{d}=A%
%TCIMACRO{\U{2124} }%
%BeginExpansion
\mathbb{Z}
%EndExpansion
^{d}.$ Thus, the unitary dual of $N$ is isomorphic to the commutative group
\[
\frac{\widehat{%
%TCIMACRO{\U{211d} }%
%BeginExpansion
\mathbb{R}
%EndExpansion
^{d}}}{A^{-tr}%
%TCIMACRO{\U{2124} }%
%BeginExpansion
\mathbb{Z}
%EndExpansion
^{d}}\times\frac{\widehat{%
%TCIMACRO{\U{211d} }%
%BeginExpansion
\mathbb{R}
%EndExpansion
^{d}}}{B^{-tr}%
%TCIMACRO{\U{2124} }%
%BeginExpansion
\mathbb{Z}
%EndExpansion
^{d}}\times\widehat{%
%TCIMACRO{\U{2124} }%
%BeginExpansion
\mathbb{Z}
%EndExpansion
_{m}}.
\]
Letting $\Lambda_{1}\subset\widehat{%
%TCIMACRO{\U{211d} }%
%BeginExpansion
\mathbb{R}
%EndExpansion
^{d}}$ be a measurable fundamental domain for $A^{-tr}%
%TCIMACRO{\U{2124} }%
%BeginExpansion
\mathbb{Z}
%EndExpansion
^{d}$ and $\Lambda_{2}\subset\widehat{%
%TCIMACRO{\U{211d} }%
%BeginExpansion
\mathbb{R}
%EndExpansion
^{d}}$ a measurable fundamental domain for $B^{-tr}%
%TCIMACRO{\U{2124} }%
%BeginExpansion
\mathbb{Z}
%EndExpansion
^{d},$ the unitary dual of $N$ is parameterized by the set
\[
\left\{  \left(  \gamma_{1},\gamma_{2},\sigma\right)  :\gamma_{1}\in
\Lambda_{1},\gamma_{2}\in\Lambda_{2},\sigma\in\widehat{%
%TCIMACRO{\U{2124} }%
%BeginExpansion
\mathbb{Z}
%EndExpansion
_{m}}\right\}  .
\]
Next, the action of a fixed character of $N$ is computed as follows.
\[
\chi_{\left(  \gamma_{1},\gamma_{2},\sigma\right)  }\left(  T_{k}M_{l}e^{2\pi
i\theta}\right)  =\exp\left[  2\pi i\left\langle \gamma_{1},k\right\rangle
\right]  \exp\left[  2\pi i\left\langle \gamma_{2},l\right\rangle \right]
\exp\left[  2\pi i\sigma\theta\right]  .
\]

We recall one important result due to Mackey which is also presented in
\cite{Lipsman2}. We will need this proposition to compute the unitary dual of
the group $G.$

\begin{proposition}
\label{Mackey}Let $N$ be a normal subgroup of $G.$ Assume that $N$ is type $I$
and is regularly embedded. Let $\pi$ be an arbitrary element of the unitary
dual of $N.$ $G$ acts on the unitary dual of $N$ as follows:%
\[
x\cdot\pi\left(  y\right)  =\pi\left(  x^{-1}yx\right)  ,x\in G,y\in N.
\]
Let $G_{\pi}$ be the stabilizer group of the $G$-action on $\pi.$%
\[
\widehat{G}=%
%TCIMACRO{\dbigcup \limits_{\pi\in\widehat{N}/G}}%
%BeginExpansion
{\displaystyle\bigcup\limits_{\pi\in\widehat{N}/G}}
%EndExpansion
\left\{  \mathrm{Ind}_{G_{\pi}}^{G}\left(  \upsilon\right)  :\upsilon
\in\widehat{G_{\pi}}\text{ and }\upsilon|N\text{ is a multiple of }%
\pi\right\}  .
\]

\end{proposition}

Now, we will apply Proposition \ref{Mackey} to
\[
G=\left\langle T_{k},M_{l}:k\in%
%TCIMACRO{\U{2124} }%
%BeginExpansion
\mathbb{Z}
%EndExpansion
^{d},l\in B%
%TCIMACRO{\U{2124} }%
%BeginExpansion
\mathbb{Z}
%EndExpansion
^{d}\right\rangle .
\]
First, we recall from Lemma \ref{finite}, that $G/N$ is a finite group, and
thus $G$ is a compact extension of an abelian group. Referring to
\cite{Lipsman1} I Chapter 4, $G$ is type $I$ and $N$ is regularly embedded.
Let $P\in G$, $k\in B^{-tr}%
%TCIMACRO{\U{2124} }%
%BeginExpansion
\mathbb{Z}
%EndExpansion
^{d}\cap%
%TCIMACRO{\U{2124} }%
%BeginExpansion
\mathbb{Z}
%EndExpansion
^{d},$ $l\in B%
%TCIMACRO{\U{2124} }%
%BeginExpansion
\mathbb{Z}
%EndExpansion
^{d}$ and let $\chi_{\left(  \gamma_{1},\gamma_{2},\sigma\right)  }$ be a
character of $N.$ We define the action of $G$ on the unitary dual of $N$
multiplicatively such that for $P\in G$,%
\begin{equation}
P\cdot\chi_{\left(  \gamma_{1},\gamma_{2},\sigma\right)  }\left(  T_{k}%
M_{l}e^{2\pi i\theta}\right)  =\chi_{\left(  \gamma_{1},\gamma_{2}%
,\sigma\right)  }\left(  P^{-1}\left(  T_{k}M_{l}e^{2\pi i\theta}\right)
P\right)  . \label{coadjoint}%
\end{equation}

\begin{definition}
We define a measurable map $\rho:\widehat{%
%TCIMACRO{\U{211d} }%
%BeginExpansion
\mathbb{R}
%EndExpansion
^{d}}\rightarrow\Lambda_{2}$ such that $\rho\left(  x\right)  =y_{x}$ if and
only $y_{x}$ is the unique element in $\Lambda_{2}$ such that $x=y_{x}+l,$
where $l\in B^{-tr}%
%TCIMACRO{\U{2124} }%
%BeginExpansion
\mathbb{Z}
%EndExpansion
^{d}$.
\end{definition}

Since the collection of sets
\[
\left\{  \Lambda_{2}+j:j\in B^{-tr}%
%TCIMACRO{\U{2124} }%
%BeginExpansion
\mathbb{Z}
%EndExpansion
^{d}\right\}
\]
is a measurable partition of $\widehat{%
%TCIMACRO{\U{211d} }%
%BeginExpansion
\mathbb{R}
%EndExpansion
^{d}}$ then the map $\rho$ makes sense.

\begin{lemma}
\label{orbit}For $s\in%
%TCIMACRO{\U{2124} }%
%BeginExpansion
\mathbb{Z}
%EndExpansion
^{d},l\in B%
%TCIMACRO{\U{2124} }%
%BeginExpansion
\mathbb{Z}
%EndExpansion
^{d}$ and $e^{2\pi i\theta}\in\left[  G,G\right]  $ we have

\begin{enumerate}
\item $T_{s}\cdot\chi_{\left(  \gamma_{1},\gamma_{2},\sigma\right)  }%
=\chi_{\left(  \gamma_{1},\rho\left(  \gamma_{2}+\sigma s\right)
,\sigma\right)  }$

\item $M_{l}\cdot\chi_{\left(  \gamma_{1},\gamma_{2},\sigma\right)  }%
=\chi_{\left(  \gamma_{1},\gamma_{2},\sigma\right)  }$

\item $e^{2\pi i\theta}\cdot\chi_{\left(  \gamma_{1},\gamma_{2},\sigma\right)
}=\chi_{\left(  \gamma_{1},\gamma_{2},\sigma\right)  }$
\end{enumerate}
\end{lemma}

\begin{proof}
Since $e^{2\pi i\theta}$ is a central element of $G$ then clearly part (c) is
correct. Now, for part (a), let $T_{k},M_{l},$ such that $k\in B^{-tr}%
%TCIMACRO{\U{2124} }%
%BeginExpansion
\mathbb{Z}
%EndExpansion
^{d}\cap%
%TCIMACRO{\U{2124} }%
%BeginExpansion
\mathbb{Z}
%EndExpansion
^{d},l\in B%
%TCIMACRO{\U{2124} }%
%BeginExpansion
\mathbb{Z}
%EndExpansion
^{d}$ and $s\in%
%TCIMACRO{\U{2124} }%
%BeginExpansion
\mathbb{Z}
%EndExpansion
^{d}$%
\begin{align*}
&  T_{s}\cdot\chi_{\left(  \gamma_{1},\gamma_{2},\sigma\right)  }\left(
  T_{k}M_{l}e^{2\pi i\theta}\right) \\
&  =\chi_{\left(  \gamma_{1},\gamma_{2},\sigma\right)  }\left(  T_{s}%
^{-1}\left(  T_{k}M_{l}e^{2\pi i\theta}\right)  T_{s}\right) \\
&  =\chi_{\left(  \gamma_{1},\gamma_{2},\sigma\right)  }\left(  T_{k}%
M_{l}e^{2\pi i\left(  \theta+\left\langle l,s\right\rangle \right)  }\right)
\\&  =\chi_{\left(  \gamma_{1},\gamma_{2}+s\sigma,\sigma\right)  }\left(
T_{k}M_{l}e^{2\pi i\theta}\right)  .
\end{align*}
Thus, $T_{s}\cdot\chi_{\left(  \gamma_{1},\gamma_{2},\sigma\right)  }%
=\chi_{\left(  \gamma_{1},\rho\left(  \gamma_{2}+\sigma s\right)
,\sigma\right)  }.$ For part (b), since $k\in B^{-tr}%
%TCIMACRO{\U{2124} }%
%BeginExpansion
\mathbb{Z}
%EndExpansion
^{d}\cap%
%TCIMACRO{\U{2124} }%
%BeginExpansion
\mathbb{Z}
%EndExpansion
^{d},l\in B%
%TCIMACRO{\U{2124} }%
%BeginExpansion
\mathbb{Z}
%EndExpansion
^{d}$
\begin{align*}
\chi_{\left(  \gamma_{1},\gamma_{2},\sigma\right)  }\left(  M_{-l}\left(
T_{k}M_{l}e^{2\pi i\theta}\right)  M_{l}\right)   &  =\chi_{\left(  \gamma
_{1},\gamma_{2},\sigma\right)  }\left(  M_{-l}M_{l}T_{k}M_{l}e^{2\pi i\theta
}e^{2\pi i\left(  \left\langle l,k\right\rangle \right)  }\right) \\
&  =\chi_{\left(  \gamma_{1},\gamma_{2},\sigma\right)  }\left(  T_{k}%
M_{l}e^{2\pi i\theta}\right)  .
\end{align*}

\end{proof}

\begin{lemma}
\label{functional}The stabilizer group of a fixed character $\chi_{\left(
\gamma_{1},\gamma_{2},\sigma\right)  }$ in the unitary dual of $N$ is given by%
\begin{equation}
G_{\left(  \gamma_{1},\gamma_{2},\sigma\right)  }=\left\langle
\begin{array}
[c]{c}%
T_{k}M_{l}e^{2\pi i\theta}\in G\text{ }|\text{ }\sigma k=B^{-tr}j\text{ for
some }j\in%
%TCIMACRO{\U{2124} }%
%BeginExpansion
\mathbb{Z}
%EndExpansion
^{d}\\
l\in B%
%TCIMACRO{\U{2124} }%
%BeginExpansion
\mathbb{Z}
%EndExpansion
^{d},e^{2\pi i\theta}\in\left[  G,G\right]
\end{array}
\right\rangle \label{stab}%
\end{equation}

\end{lemma}

Thanks to (\ref{stab}), we may write $
G_{\left(  \gamma_{1},\gamma_{2},\sigma\right)  }=A\left(  \sigma\right)
%TCIMACRO{\U{2124} }%
%BeginExpansion
\mathbb{Z}
%EndExpansion
^{d}\times B%
%TCIMACRO{\U{2124} }%
%BeginExpansion
\mathbb{Z}
%EndExpansion
^{d}\times\left[  G,G\right]$ where $A\left(  \sigma\right)  $ is a full-rank matrix of order $d$ and
$A\left(  \sigma\right)
%TCIMACRO{\U{2124} }%
%BeginExpansion
\mathbb{Z}
%EndExpansion
^{d}\geq A%
%TCIMACRO{\U{2124} }%
%BeginExpansion
\mathbb{Z}
%EndExpansion
^{d}.$ For a fixed element $\left(  \gamma_{1},\gamma_{2},\sigma\right)  $ in
the unitary dual of $N,$ we define the set%
\begin{equation}
\mathcal{U}_{\sigma}=\Lambda\left(  \sigma\right)  \times\Lambda_{2}%
\times\widehat{\left[  G,G\right]  } \label{SetU}%
\end{equation}
where $\Lambda\left(  \sigma\right)  $ is a fundamental domain for
$\frac{\widehat{%
%TCIMACRO{\U{211d} }%
%BeginExpansion
\mathbb{R}
%EndExpansion
^{d}}}{A\left(  \sigma\right)  ^{-tr}%
%TCIMACRO{\U{2124} }%
%BeginExpansion
\mathbb{Z}
%EndExpansion
^{d}}$ such that $\Lambda_{1}$ is contained in $\Lambda\left(  \sigma\right)
.$ Since we need Proposition \ref{Mackey} to compute the unitary dual of $G,$
we would like to be able to assert that if $\chi_{\left(  \gamma_{1}%
,\gamma_{2},\sigma\right)  }$ is a character of the group $N,$ it is possible
to extend $\chi_{\left(  \gamma_{1},\gamma_{2},\sigma\right)  }$ to a
character of the stabilizer group $G_{\left(  \gamma_{1},\gamma_{2}%
,\sigma\right)  }.$ However, we will need a few lemmas first.

\begin{lemma}
\label{EXT}Let $\lambda=\left(  \eta,\gamma_{2},\sigma\right)  \in
\mathcal{U}_{\sigma}$ such that $\eta=\gamma_{1}+A^{-tr}j$ for some $j\in%
%TCIMACRO{\U{2124} }%
%BeginExpansion
\mathbb{Z}
%EndExpansion
^{d}.$ If $\chi_{\left(  \gamma_{1}+A^{-tr}j,\gamma_{2},\sigma\right)  }$ is a
character of $G_{\left(  \gamma_{1},\gamma_{2},\sigma\right)  },$ then
\[
\chi_{\left(  \gamma_{1}+A^{-tr}j,\gamma_{2},\sigma\right)  }\left(
T_{k}M_{l}e^{2\pi i\theta}\right)  =\chi_{\left(  \gamma_{1},\gamma_{2}%
,\sigma\right)  }\left(  T_{k}M_{l}e^{2\pi i\theta}\right)
\]

\end{lemma}

\begin{proof}
Let $T_{k}M_{l}e^{2\pi i\theta}\in N.$ Since $k\in A%
%TCIMACRO{\U{2124} }%
%BeginExpansion
\mathbb{Z}
%EndExpansion
^{d},$ there exists $k^{\prime}\in%
%TCIMACRO{\U{2124} }%
%BeginExpansion
\mathbb{Z}
%EndExpansion
^{d}$ such that $k=Ak^{\prime}$
\begin{align*}
\chi_{\lambda}\left(  T_{k}M_{l}e^{2\pi i\theta}\right)   &  =\exp\left(  2\pi i\left\langle \gamma_{1},k\right\rangle \right)
\exp\left(  2\pi i\left\langle j,k^{\prime}\right\rangle \right)  \exp\left(
2\pi i\left\langle \gamma_{2},l\right\rangle \right)  e^{2\pi i\sigma\theta}\\
&  =\exp\left(  2\pi i\left\langle \gamma_{1},k\right\rangle \right)
\exp\left(  2\pi i\left\langle \gamma_{2},l\right\rangle \right)  e^{2\pi
i\sigma\theta}\\
&  =\chi_{\left(  \gamma_{1},\gamma_{2},\sigma\right)  }\left(  T_{k}%
M_{l}e^{2\pi i\theta}\right)  .
\end{align*}

\end{proof}

\begin{lemma}
\label{X1}For a fixed $
\left(  \gamma_{1},\gamma_{2},\sigma\right)  \in\Lambda_{1}\times\Lambda
_{2}\times\widehat{%
%TCIMACRO{\U{2124} }%
%BeginExpansion
\mathbb{Z}
%EndExpansion
_{m}}$ in the unitary dual of $N,$ if $\lambda=\left(  \eta,\gamma_{2},\sigma\right)
\in\mathcal{U}_{\sigma},\eta=\gamma_{1}+A^{-tr}j$ for some $j\in%
%TCIMACRO{\U{2124} }%
%BeginExpansion
\mathbb{Z}
%EndExpansion
^{d}$ then $
\chi_{\lambda}\left[  G_{\left(  \gamma_{1},\gamma_{2},\sigma\right)
},G_{\left(  \gamma_{1},\gamma_{2},\sigma\right)  }\right]  =1.$
\end{lemma}

\begin{proof}
Let $T_{k}M_{l}\tau\in G_{\left(  \gamma_{1},\gamma_{2},\sigma\right)  }.$ It
suffices to check that
\[
\chi_{\lambda}\left(  T_{k}M_{l}\tau T_{k}^{-1}M_{l}^{-1}\tau^{-1}\right)  =1
\]
where $\tau\in\left[  G,G\right]  .$ First, we observe that
\[
\chi_{\lambda}\left(  T_{k}M_{l}\tau T_{k}^{-1}M_{l}^{-1}\tau^{-1}\right)
=\chi_{\lambda}\left(  T_{k}M_{l}T_{k}^{-1}M_{l}^{-1}\right)  =\exp\left[
-2\pi i\left\langle \sigma k,l\right\rangle \right]  .
\]
Applying Lemma \ref{functional}, since $T_{k}\in G_{\left(  \gamma_{1}%
,\gamma_{2},\sigma\right)  }$ there exists some $p\in%
%TCIMACRO{\U{2124} }%
%BeginExpansion
\mathbb{Z}
%EndExpansion
^{d}$ such that
\[
\chi_{\lambda}\left(  T_{k}M_{l}\tau T_{k}^{-1}M_{l}^{-1}\tau^{-1}\right)
=\exp\left[  2\pi i\sigma\left\langle B^{-tr}p,l\right\rangle \right]  .
\]
Secondly, using the fact that $l\in B%
%TCIMACRO{\U{2124} }%
%BeginExpansion
\mathbb{Z}
%EndExpansion
^{d}$ there exists $l^{\prime}\in%
%TCIMACRO{\U{2124} }%
%BeginExpansion
\mathbb{Z}
%EndExpansion
^{d}$ such that
\[
\chi_{\lambda}\left(  T_{k}M_{l}\tau T_{k}^{-1}M_{l}^{-1}\tau^{-1}\right)
=\exp\left[  2\pi i\sigma\left\langle B^{-tr}p,Bl^{\prime}\right\rangle
\right]  =1.
\]

\end{proof}

The following lemma allows us to establish the extension of characters from
the normal subgroup $N$ to the stabilizer groups.

\begin{lemma}
\label{extension character}Fix $\left(  \gamma_{1},\gamma_{2},\sigma\right)  $
in the unitary dual of $N.$ Let $\lambda=\left(  \eta,\gamma_{2}%
,\sigma\right)  \in\mathcal{U}_{\sigma}$. Then $\chi_{\lambda}$ defines a
character on $G_{\left(  \gamma_{1},\gamma_{2},\sigma\right)  }$.
\end{lemma}

\begin{proof}
Given $T_{k_{1}}M_{l_{1}}e^{2\pi i\theta_{1}},$ and $T_{k_{2}}M_{l_{2}}e^{2\pi
i\theta_{2}}\in G_{\left(  \gamma_{1},\gamma_{2},\sigma\right)  },$ it is easy
to see that
\[
\left(  T_{k_{1}}M_{l_{1}}e^{2\pi i\theta_{1}}\right)  \left(  T_{k_{2}%
}M_{l_{2}}e^{2\pi i\theta_{2}}\right)  =T_{k_{1}+k_{2}}M_{l_{1}+l_{2}}e^{2\pi
i\left(  \theta_{1}+\theta_{2}\right)  }e^{2\pi i\left\langle l_{1}%
,k_{2}\right\rangle }%
\]
where $e^{2\pi i\left\langle l_{1},k_{2}\right\rangle }\in\left[  G_{\left(
\gamma_{1},\gamma_{2},\sigma\right)  },G_{\left(  \gamma_{1},\gamma_{2}%
,\sigma\right)  }\right]  .$ We want to show that $\chi_{\lambda}$ defines a
homomorphism from $G_{\left(  \gamma_{1},\gamma_{2},\sigma\right)  }$ into the
circle group. Since $\chi_{\lambda}\left[  G_{\left(  \gamma_{1},\gamma
_{2},\sigma\right)  },G_{\left(  \gamma_{1},\gamma_{2},\sigma\right)
}\right]  =1$ then
\begin{align*}
&  \chi_{\lambda}\left(  \left(  T_{k_{1}}M_{l_{1}}e^{2\pi i\theta_{1}%
}\right)  \left(  T_{k_{2}}M_{l_{2}}e^{2\pi i\theta_{2}}\right)  \right) \\
&  =\chi_{\lambda}\left(  T_{k_{1}}T_{k_{2}}M_{l_{1}}M_{l_{2}}e^{2\pi
i\theta_{2}}e^{2\pi i\theta_{1}}e^{2\pi i\left\langle l_{1},k_{2}\right\rangle
}\right) \\
&  =\chi_{\lambda}\left(  T_{k_{1}+k_{2}}M_{l_{1}+l_{2}}e^{2\pi i\left(
\theta_{1}+\theta_{2}+\left\langle l_{1},k_{2}\right\rangle \right)  }\right)
\\
&  =\exp\left(  2\pi i\left\langle \eta,k_{1}+k_{2}\right\rangle \right)
\exp\left(  2\pi i\left\langle \gamma_{2},l_{1}+l_{2}\right\rangle \right)
e^{2\pi i\sigma\left(  \theta_{1}+\theta_{2}\right)  }e^{2\pi i\left(
\sigma\left\langle l_{1},k_{2}\right\rangle \right)  }%
\end{align*}
From Lemma \ref{X1}, $e^{2\pi i\left(  \sigma\left\langle l_{1},k_{2}%
\right\rangle \right)  }=1$ and
\begin{align*}
&  \chi_{\lambda}\left(  \left(  T_{k_{1}}M_{l_{1}}e^{2\pi i\theta_{1}%
}\right)  \left(  T_{k_{2}}M_{l_{2}}e^{2\pi i\theta_{2}}\right)  \right) \\
&  =\exp\left(  2\pi i\left\langle \eta,k_{1}+k_{2}\right\rangle \right)
\exp\left(  2\pi i\left\langle \gamma_{2},l_{1}+l_{2}\right\rangle \right)
e^{2\pi i\sigma\left(  \theta_{1}+\theta_{2}\right)  }%
\end{align*}
Now, using Lemma \ref{EXT}%
\[
\chi_{\lambda}\left(  \left(  T_{k_{1}}M_{l_{1}}e^{2\pi i\theta_{1}}\right)
\left(  T_{k_{2}}M_{l_{2}}e^{2\pi i\theta_{2}}\right)  \right)  =\chi
_{\lambda}\left(  T_{k_{1}}M_{l_{1}}e^{2\pi i\theta_{1}}\right)  \chi
_{\lambda}\left(  T_{k_{2}}M_{l_{2}}e^{2\pi i\theta_{2}}\right)  .
\]
Thus $\chi_{\lambda}$ defines a character on $G_{\left(  \gamma_{1},\gamma
_{2},\sigma\right)  }.$
\end{proof}

\begin{remark}
Fix $\left(  \gamma_{1},\gamma_{2},\sigma\right)  $ in the unitary dual of
$N.$ Let $\eta=\gamma_{1}+A^{-tr}j\in\mathcal{U}_{\sigma}$ where $j\in%
%TCIMACRO{\U{2124} }%
%BeginExpansion
\mathbb{Z}
%EndExpansion
^{d}.$ The character $\chi_{\left(  \eta,\gamma_{2},\sigma\right)  }$ is
called an extension of $\chi_{\left(  \gamma_{1},\gamma_{2},\sigma\right)  }$
from $N$ to the stabilizer group $G_{\left(  \gamma_{1},\gamma_{2}%
,\sigma\right)  }.$
\end{remark}

\begin{proposition}
The unitary dual of $G$ is parameterized by the set%
\[
\Sigma=%
%TCIMACRO{\dbigcup \limits_{\lambda\in\Omega}}%
%BeginExpansion
{\displaystyle\bigcup\limits_{\lambda\in\Omega}}
%EndExpansion
\widehat{G_{\lambda}}\text{ where }\Omega=%
%TCIMACRO{\dbigcup \limits_{\sigma\in\widehat{\left[  G,G\right]  }}}%
%BeginExpansion
{\displaystyle\bigcup\limits_{\sigma\in\widehat{\left[  G,G\right]  }}}
%EndExpansion
\left(  \Lambda_{1}\times\mathbf{E}_{\sigma}\times\left\{  \sigma\right\}
\right)  ,
\]
and $\mathbf{E}_{\sigma}$ is a measurable subset of $\Lambda_{2}$ satisfying
the condition
\[%
%TCIMACRO{\dbigcup \limits_{s\in\mathbb{Z}^{d}}}%
%BeginExpansion
{\displaystyle\bigcup\limits_{s\in\mathbb{Z}^{d}}}
%EndExpansion
\rho\left(  \mathbf{E}_{\sigma}+\sigma s\right)  =\Lambda_{2}.
\]

\end{proposition}

\begin{proof}
Fixing $\sigma\in\widehat{\left[  G,G\right]  },$ recall that
\[
G\cdot\left(  \gamma_{1},\gamma_{2},\sigma\right)  =\left\{  \left(
\gamma_{1},\rho\left(  \gamma_{2}+\sigma s\right)  ,\sigma\right)  :s\in%
%TCIMACRO{\U{2124} }%
%BeginExpansion
\mathbb{Z}
%EndExpansion
^{d}\right\}  .
\]
For a fixed $\sigma\in\widehat{\left[  G,G\right]  },$ we pick a measurable
set $\mathbf{E}_{\sigma}\subset$ $\Lambda_{2}$ such that
\begin{equation}%
%TCIMACRO{\dbigcup \limits_{s\in\mathbb{Z}^{d}}}%
%BeginExpansion
{\displaystyle\bigcup\limits_{s\in\mathbb{Z}^{d}}}
%EndExpansion
\rho\left(  \mathbf{E}_{\sigma}+\sigma s\right)  =\Lambda_{2}. \label{union}%
\end{equation}
The set
\[
\Omega=%
%TCIMACRO{\dbigcup \limits_{\sigma\in\widehat{\left[  G,G\right]  }}}%
%BeginExpansion
{\displaystyle\bigcup\limits_{\sigma\in\widehat{\left[  G,G\right]  }}}
%EndExpansion
\left(  \Lambda_{1}\times\mathbf{E}_{\sigma}\times\left\{  \sigma\right\}
\right)
\]
parameterizes the orbit space $\widehat{N}/G.$ Thus, following Mackey's result
(see Proposition \ref{Mackey}), the unitary dual of $G$ is parametrized by the set
\[
\Sigma=%
%TCIMACRO{\dbigcup \limits_{\left(  \gamma_{1},\gamma_{2},\sigma\right)
%\in\Omega}}%
%BeginExpansion
{\displaystyle\bigcup\limits_{\left(  \gamma_{1},\gamma_{2},\sigma\right)
\in\Omega}}
%EndExpansion
\widehat{G_{\left(  \gamma_{1},\gamma_{2},\sigma\right)  }}.
\]

\end{proof}

Following the description given in \cite{Lipsman1}, Section $4$, let
$\chi_{\left(  \gamma_{1},\gamma_{2},\sigma\right)  }\in\widehat{N}$ and let
$\chi_{\left(  \gamma_{1},\gamma_{2},\sigma\right)  }^{j}=\chi_{\left(
\gamma_{1}+A^{-tr}j,\gamma_{2},\sigma\right)  }$ be its extended
representation from $N$ to $G_{\left(  \gamma_{1},\gamma_{2},\sigma\right)
}.$ Let $\zeta$ be an irreducible representation of $G_{\left(  \gamma
_{1},\gamma_{2},\sigma\right)  }/N$ and define its lift to $G_{\left(
\gamma_{1},\gamma_{2},\sigma\right)  }$ which we denote by $\overline{\zeta}.$
For a fixed $\left(  \gamma_{1},\gamma_{2},\sigma,\zeta\right)  $ we define
the representation
\[
\rho_{\left(  \gamma_{1},\gamma_{2},\sigma,\zeta\right)  }=\mathrm{Ind}%
_{G_{\left(  \gamma_{1},\gamma_{2},\sigma\right)  }}^{G}\left(  \chi_{\left(
\gamma_{1},\gamma_{2},\sigma\right)  }^{j}\otimes\overline{\zeta}\right)
\]
acting in the Hilbert space
\begin{equation}
\mathcal{H}_{\left(  \gamma_{1},\gamma_{2},\sigma,\zeta\right)  }=\left\{
\begin{array}
[c]{c}%
\mathbf{u}:G\rightarrow%
%TCIMACRO{\U{2102} }%
%BeginExpansion
\mathbb{C}
%EndExpansion
:\\
\mathbf{u}\left(  T_{k}P\right)  =\left[  \left(  \chi_{\left(  \gamma
_{1},\gamma_{2},\sigma\right)  }^{j}\otimes\overline{\zeta}\right)  \left(
P\right)  \right]  ^{-1}\mathbf{u}\left(  T_{k}\right)  ,\\
\text{where }P\in G_{\left(  \gamma_{1},\gamma_{2},\sigma\right)  }%
\end{array}
\right\}  \label{space}%
\end{equation}
which is naturally identified with
\[
l^{2}\left(  G/G_{\left(  \gamma_{1},\gamma_{2},\sigma\right)  }\right)  \cong%
%TCIMACRO{\U{2102} }%
%BeginExpansion
\mathbb{C}
%EndExpansion
^{\mathrm{card}\left(  G/G_{\left(  \gamma_{1},\gamma_{2},\sigma\right)
}\right)  }.
\]
The inner product on $\mathcal{H}_{\left(  \gamma_{1},\gamma_{2},\sigma
,\zeta\right)  }$ is given by
\[
\left\langle \mathbf{u,v}\right\rangle _{\mathcal{H}_{\left(  \gamma
_{1},\gamma_{2},\sigma,\zeta\right)  }}=\sum_{PG_{\left(  \gamma_{1}%
,\gamma_{2},\sigma\right)  }\in G/G_{\left(  \gamma_{1},\gamma_{2}%
,\sigma\right)  }}\mathbf{u}\left(  P\right)  \overline{\mathbf{v}\left(
P\right)  }%
\]
Notice that the number of elements in $G/G_{\left(  \gamma_{1},\gamma
_{2},\sigma\right)  }$ is bounded above by the order of the finite group
\[
G/N\cong\frac{%
%TCIMACRO{\U{2124} }%
%BeginExpansion
\mathbb{Z}
%EndExpansion
^{d}}{A%
%TCIMACRO{\U{2124} }%
%BeginExpansion
\mathbb{Z}
%EndExpansion
^{d}}%
\]
which has precisely $\left\vert \det A\right\vert $ elements.

\begin{remark}
Every irreducible representation of $G$ is monomial. That is, every
irreducible representation of $G$ is induced by a one-dimensional
representation of some subgroup of $G.$ Given $\mathbf{u}$ $\mathbf{\in}$
$\mathcal{H}_{\left(  \gamma_{1},\gamma_{2},\sigma,\zeta\right)  },$
\[
\rho_{\left(  \gamma_{1},\gamma_{2},\sigma,\zeta\right)  }\left(  T_{k}%
M_{l}e^{2\pi i\theta}\right)  \mathbf{u}\left(  T_{s}\right)  =\mathbf{u}%
\left(  \left(  T_{k}M_{l}e^{2\pi i\theta}\right)  ^{-1}T_{s}\right)
\]
and $\mathbf{u}\left(  \left(  T_{k}M_{l}e^{2\pi i\theta}\right)  ^{-1}%
T_{s}\right)  $ is computed by following the rules defined in (\ref{space})
where
\[
T_{s}\in\left\{  T_{k_{1}},\cdots,T_{k_{\left\vert \det A\right\vert }%
}\right\}
\]
and
\[
\left\{  k_1+A\mathbb{Z}
%EndExpansion
^{d},\cdots,k_{\left\vert \det A\right\vert }+A\mathbb{Z}
%EndExpansion
^{d}\right\}
\]
is a set of representative elements of the quotient group $\frac{%
%TCIMACRO{\U{2124} }%
%BeginExpansion
\mathbb{Z}
%EndExpansion
^{d}}{A%
%TCIMACRO{\U{2124} }%
%BeginExpansion
\mathbb{Z}
%EndExpansion
^{d}}.$
\end{remark}

The lemma above is a standard computation of an induced representation. The
interested reader is referred to \cite{Folland}

Next, we define the set
\[
\Sigma_{\sigma}=\left\{
\begin{array}
[c]{c}%
\mathrm{Ind}_{G_{\left(  \gamma_{1},\gamma_{2},\sigma\right)  }}^{G}\left(
\chi_{\left(  \gamma_{1},\gamma_{2},\sigma\right)  }^{j}\otimes\overline
{\zeta}\right)  :\zeta\in\widehat{G_{\left(  \gamma_{1},\gamma_{2}%
,\sigma\right)  }/N}\\
\text{ }\left(  \gamma_{1},\gamma_{2},\sigma\right)  \in\Lambda_{1}%
\times\mathbf{E}_{\sigma}\times\left\{  \sigma\right\}
\end{array}
\right\}  .
\]
A more convenient description of the unitary dual of $G$ which will be helpful
when we compute the Plancherel measure is
\begin{equation}
\Sigma=%
%TCIMACRO{\dbigcup \limits_{\sigma\in\widehat{\left[  G,G\right]  }}}%
%BeginExpansion
{\displaystyle\bigcup\limits_{\sigma\in\widehat{\left[  G,G\right]  }}}
%EndExpansion
\Sigma_{\sigma}. \label{spectrum}%
\end{equation}

Now that we have a complete description of the unitary dual of the group $G$,
we will provide a computation of its \textbf{Plancherel measure. }

\begin{theorem}
\label{fibered measure}The Plancherel measure is a \textbf{fiber measure}
which is given by
\[
d\mu\left(  \rho_{\left(  \gamma_{1},\gamma_{2},\sigma,\zeta\right)  }\right)
=\frac{dm_{1}\left(  \gamma_{1}\right)  dm_{2}\left(  \gamma_{2}\right)
dm_{3}\left(  \sigma\right)  dm_{4}\left(  \zeta\right)  }{\left\vert \det
A\right\vert ^{-1}m_{2}\left(  \mathbf{E}_{\sigma}\right)  }.
\]
The measures $dm_{1}$, $dm_{2}$ are the canonical Lebesgue measures defined on
$\Lambda_{1}$ and $\Lambda_{2}$ respectively. The measure $dm_{3}$ is the
counting measure on the unitary dual of the commutator $\left[  G,G\right]  $
and $dm_{4}$ is the counting measure on the dual of the little group
$\widehat{G_{\left(  \gamma_{1},\gamma_{2},\sigma\right)  }/N}$ with weight
$\left(  G_{\left(  \gamma_{1},\gamma_{2},\sigma\right)  }:N\right)  .$
Moreover if $L$ is the left regular representation of $G,$ the direct integral
decompositon of $L$ into irreducible representations of $G$ is
\[
\int_{\Sigma}^{\oplus}\rho_{\left(  \gamma_{1},\gamma_{2},\sigma,\zeta\right)
}\otimes\mathbf{1}_{%
%TCIMACRO{\U{2102} }%
%BeginExpansion
\mathbb{C}
%EndExpansion
^{\mathbf{n}\left(  \gamma_{1},\gamma_{2},\sigma,\zeta\right)  }}d\mu\left(
\rho_{\left(  \gamma_{1},\gamma_{2},\sigma,\zeta\right)  }\right)
\]
acting in
\[
\int_{\Sigma}^{\oplus}%
%TCIMACRO{\U{2102} }%
%BeginExpansion
\mathbb{C}
%EndExpansion
^{\mathbf{n}\left(  \gamma_{1},\gamma_{2},\sigma,\zeta\right)  }\otimes%
%TCIMACRO{\U{2102} }%
%BeginExpansion
\mathbb{C}
%EndExpansion
^{\mathbf{n}\left(  \gamma_{1},\gamma_{2},\sigma,\zeta\right)  }d\mu\left(
\rho_{\left(  \gamma_{1},\gamma_{2},\sigma,\zeta\right)  }\right)
\]
and
\[
\mathbf{n}\left(  \gamma_{1},\gamma_{2},\sigma,\zeta\right)  =\mathrm{card}%
\left(  G/G_{\left(  \gamma_{1},\gamma_{2},\sigma\right)  }\right)  .
\]

\end{theorem}

\begin{proof}
The theorem above is an application of the abstract case given in
\cite{Lipsman1} II (Theorem 2.3), and the precise weight of the Plancherel
measure is obtained by some normalization.
\end{proof}

Let us suppose that $B\ $is in $GL\left(  d,%
%TCIMACRO{\U{211a} }%
%BeginExpansion
\mathbb{Q}
%EndExpansion
\right)  -GL\left(  d,%
%TCIMACRO{\U{2124} }%
%BeginExpansion
\mathbb{Z}
%EndExpansion
\right)  .$ Let $\varphi$ be any type I representation of $G.$ Then, there is
a unique measure, $\overline{\mu}$ defined on the spectral set $\Sigma$ such
that $\varphi$ is unitarily equivalent to
\begin{equation}
\int_{\Sigma}^{\oplus}\rho_{\left(  \gamma_{1},\gamma_{2},\sigma,\zeta\right)
}\otimes1_{%
%TCIMACRO{\U{2102} }%
%BeginExpansion
\mathbb{C}
%EndExpansion
^{\mathbf{m}\left(  \gamma_{1},\gamma_{2},\sigma,\zeta\right)  }}%
d\overline{\mu}\left(  \rho_{\left(  \gamma_{1},\gamma_{2},\sigma
,\zeta\right)  }\right)  \label{dir}%
\end{equation}
acting in
\[
\int_{\Sigma}^{\oplus}%
%TCIMACRO{\U{2102} }%
%BeginExpansion
\mathbb{C}
%EndExpansion
^{\mathbf{n}\left(  \gamma_{1},\gamma_{2},\sigma,\zeta\right)  }\otimes%
%TCIMACRO{\U{2102} }%
%BeginExpansion
\mathbb{C}
%EndExpansion
^{\mathbf{m}\left(  \gamma_{1},\gamma_{2},\sigma,\zeta\right)  }d\overline
{\mu}\left(  \rho_{\left(  \gamma_{1},\gamma_{2},\sigma,\zeta\right)
}\right)
\]
where $\mathbf{m}$ is the multiplicity function of the irreducible
representations occurring in the decomposition of $\varphi.$

\begin{proposition}
The representation $\varphi$ is admissible if and only if

\begin{enumerate}
\item $d\overline{\mu}\left(  \rho_{\left(  \gamma_{1},\gamma_{2},\sigma
,\zeta\right)  }\right)  $ is absolutely continuous with respect to the
Plancherel measure of $G.$

\item $\mathbf{m}\left(  \gamma_{1},\gamma_{2},\sigma,\zeta\right)  \leq$
$\mathrm{card}\left(  G/G_{\left(  \gamma_{1},\gamma_{2},\sigma\right)
}\right)  .$
\end{enumerate}
\end{proposition}

\begin{proof}
See \cite{hartmut} page 126
\end{proof}

\begin{remark}
If $B\ $is in $GL\left(  d,%
%TCIMACRO{\U{211a} }%
%BeginExpansion
\mathbb{Q}
%EndExpansion
\right)  -GL\left(  d,%
%TCIMACRO{\U{2124} }%
%BeginExpansion
\mathbb{Z}
%EndExpansion
\right)  $ and $\left\vert \det B\right\vert \leq1$ and the Gabor
representation $\pi$ (\ref{repPi}) is unitarily equivalent to (\ref{dir}) then
$\pi$ is admissible and the Proposition above is applicable.
\end{remark}

\begin{proof}
This remark is just an application of the density condition given in Lemma
\ref{density}.
\end{proof}

\section{Non-Type I Groups}

In general if $B$ has at least one non-rational entry, then the commutator
subgroup of $G$ is a infinite abelian group. In this section, we will
consider the case where $d=1,$ and
\[
G=\left\langle T_{k},M_{l}\text{ }|\text{ }k\in%
%TCIMACRO{\U{2124} }%
%BeginExpansion
\mathbb{Z}
%EndExpansion
,l\in\alpha%
%TCIMACRO{\U{2124} }%
%BeginExpansion
\mathbb{Z}
%EndExpansion
\right\rangle \text{ }%
\]
where $\alpha$ is irrational positive number. Unfortunately, the Mackey
machine will not be applicable here, and we will have to rely on different
techniques to obtain a decomposition of the left regular representation in
this case. Let $\mathbb{H}$ be the three-dimensional connected, simply
connected Heisenberg group. We define
\[
\Gamma=\left\{  \mathbf{P}_{l,m,k}=\left[
\begin{array}
[c]{ccc}%
1 & m & l\\
0 & 1 & k\\
0 & 0 & 1
\end{array}
\right]  :\left(  m,k,l\right)  \in%
%TCIMACRO{\U{2124} }%
%BeginExpansion
\mathbb{Z}
%EndExpansion
^{3}\right\}  .
\]

\begin{lemma}
There is a faithful representation $\Theta^{\alpha}$ of $\Gamma$ such that
\[
\Theta^{\alpha}\left(  \mathbf{P}_{l,0,0}\right)  =\chi_{\alpha}\left(
l\right)  ,\text{ }\Theta^{\alpha}\left(  \mathbf{P}_{0,m,0}\right)
=T_{m},\text{ and }\Theta^{\alpha}\left(  \mathbf{P}_{0,0,k}\right)
=M_{k\alpha}.
\]

\end{lemma}

\begin{proof}
$\Theta^{\alpha}$ is the restriction of an irreducible infinite-dimensional
representation of the Heisenberg group \cite{Corwin} to the lattice $\Gamma.$
Since
\[
\ker\Theta^{\alpha}=\left\{\left[
\begin{array}
[c]{ccc}%
1 & 0 & 0\\
0 & 1 & 0\\
0 & 0 & 1
\end{array}
\right]\right\}
\]
the representation $\Theta^{\alpha}$ is clearly faithful.
\end{proof}

Thus, $G\cong\Gamma$ via $\Theta^{\alpha}$ and for our purpose, it is more
convenient to work with $\Gamma.$ We define
\[
\Gamma_{1}=\left\{  \mathbf{P}_{0,m,k}:\left(  m,k\right)  \in%
%TCIMACRO{\U{2124} }%
%BeginExpansion
\mathbb{Z}
%EndExpansion
^{2}\right\}  .
\]
Let $L_{\mathbb{H}}$ be the left regular representation of the simply
connected, connected Heisenberg group
\[
\mathbb{H=}\left\{  \left[
\begin{array}
[c]{ccc}%
1 & x & z\\
0 & 1 & y\\
0 & 0 & 1
\end{array}
\right]  :\left(  z,y,x\right)  \in%
%TCIMACRO{\U{211d} }%
%BeginExpansion
\mathbb{R}
%EndExpansion
^{3}\right\}  .
\]
In fact, it is not too hard to show that $\Gamma$ is a lattice subgroup of
$\mathbb{H}$. Let $\mathcal{P}$ be the Plancherel transform of the Heisenberg
group. We recall that
\[
\mathcal{P}:L^{2}\left(  \mathbb{H}\right)  \rightarrow\int_{%
%TCIMACRO{\U{211d} }%
%BeginExpansion
\mathbb{R}
%EndExpansion
^{\ast}}^{\oplus}L^{2}\left(
%TCIMACRO{\U{211d} }%
%BeginExpansion
\mathbb{R}
%EndExpansion
\right)  \otimes L^{2}\left(
%TCIMACRO{\U{211d} }%
%BeginExpansion
\mathbb{R}
%EndExpansion
\right)  \left\vert \lambda\right\vert d\lambda
\]
where the Fourier transform is defined on $L^{2}(\mathbb{H})\cap
L^{1}(\mathbb{H})$ by
\[
\mathcal{P}\left(  f\right)  \left(  \lambda\right)  =\int_{%
%TCIMACRO{\U{211d} }%
%BeginExpansion
\mathbb{R}
%EndExpansion
^{\ast}}\pi_{\lambda}\left(  n\right)  f\left(  n\right)  dn
\]
where
\[
\pi_{\lambda}\left(  n\right)  f\left(  t\right)  =\pi_{\lambda}\left(
\mathbf{P}_{z,x,y}\right)  f\left(  t\right)  =e^{2\pi iz\lambda}e^{-2\pi
i\lambda yt}f\left(  t-x\right)  ,
\]
and the Plancherel transform is the extension of the Fourier transform to
$L^{2}(\mathbb{H})$ inducing the equality
\[
\left\Vert f\right\Vert _{L^{2}\left(  \mathbb{H}\right)  }^{2}=\int_{%
%TCIMACRO{\U{211d} }%
%BeginExpansion
\mathbb{R}
%EndExpansion
^{\ast}}\left\Vert \mathcal{P}\left(  f\right)  \left(  \lambda\right)
\right\Vert _{\mathcal{HS}}^{2}\left\vert \lambda\right\vert d\lambda.
\]
In fact, $||\cdot||_{\mathcal{HS}}$ denotes the Hilbert Schmidt norm on
$L^{2}\left(  \mathbb{R}\right)  \otimes L^{2}\left(  \mathbb{R}\right)  $.
Let $u\otimes v$ and $w\otimes y$ be rank-one operators in $L^{2}\left(
\mathbb{R}\right)  \otimes L^{2}\left(  \mathbb{R}\right)  .$ We have
\[
\left\langle u\otimes v,w\otimes y\right\rangle _{\mathcal{HS}}=\left\langle
u,w\right\rangle _{L^{2}\left(  \mathbb{R}\right)  }\left\langle
v,y\right\rangle _{L^{2}\left(  \mathbb{R}\right)  }.
\]
It is well-known that%
\[
L_{\mathbb{H}}\cong\mathcal{P\circ}\text{ }L_{\mathbb{H}}\circ\mathcal{P}%
^{-1}=\int_{%
%TCIMACRO{\U{211d} }%
%BeginExpansion
\mathbb{R}
%EndExpansion
^{\ast}}^{\oplus}\pi_{\lambda}\otimes\mathbf{1}_{L^{2}\left(  \mathbb{R}%
\right)  }\left\vert \lambda\right\vert d\lambda,
\]
where $\mathbf{1}_{L^{2}\left(  \mathbb{R}\right)  }$ is the identity operator
on $L^{2}\left(  \mathbb{R}\right)  ,$ and for a.e. $\lambda\in%
%TCIMACRO{\U{211d} }%
%BeginExpansion
\mathbb{R}
%EndExpansion
^{\ast},$
\[
\mathcal{P}(L_{\mathbb{H}}(x)\phi)(\lambda)=\pi_{\lambda}(x)\circ
\mathcal{P}\phi(\lambda).
\]
Let $\left(\mathbf{u}_{\lambda}\right)_{\lambda\in\left[  -1,0\right)  \cup\left(
0,1\right] }$ be a measurable field of unit vectors in $L^{2}\left(
%TCIMACRO{\U{211d} }%
%BeginExpansion
\mathbb{R}
%EndExpansion
\right)  .$ We define two \textbf{left-invariant multiplicity-free subspaces}
of $L^{2}\left(  \mathbb{H}\right)  $ such that
\begin{align*}
\mathbf{H}\text{ }^{+}\text{ }  &  \mathbf{=}\text{ }\mathcal{P}^{-1}\left(
\int_{\left(  0,1\right]  }^{\oplus}L^{2}\left(
%TCIMACRO{\U{211d} }%
%BeginExpansion
\mathbb{R}
%EndExpansion
\right)  \otimes\mathbf{u}_{\lambda}\left\vert \lambda\right\vert
d\lambda\right)  ,\\
\mathbf{H}\text{ }^{-}\text{ }  &  \mathbf{=}\text{ }\mathcal{P}^{-1}\left(
\int_{\left[  -1,0\right)  }^{\oplus}L^{2}\left(
%TCIMACRO{\U{211d} }%
%BeginExpansion
\mathbb{R}
%EndExpansion
\right)  \otimes\mathbf{u}_{\lambda}\left\vert \lambda\right\vert
d\lambda\right)
\end{align*}
and $\mathbf{H}$ $^{+},\mathbf{H}$ $^{-}$ are mutually orthogonal. The
following lemma has also been proved in more general terms in \cite{Oussa}.
However the proof is short enough to be presented again in this section.

\begin{lemma}
The representation $\left(  L_{\mathbb{H}}|\Gamma,\mathbf{H}\text{ }%
^{+}\right)  $ is cyclic and $\mathbf{H}$ $^{+}$ admits a Parseval frame of
the type $L_{\mathbb{H}}\left(  \Gamma\right)  f$ with $\left\Vert
f\right\Vert _{L^{2}(\mathbb{H})}^{2}=\frac{1}{2}.$
\end{lemma}

\begin{proof}
Let $f,\phi\in L^{2}(\mathbb{H})$.%
\begin{align*}
&
%TCIMACRO{\dsum \limits_{\gamma\in\Gamma}}%
%BeginExpansion
{\displaystyle\sum\limits_{\gamma\in\Gamma}}
%EndExpansion
\left\vert \left\langle \phi,L_{\mathbb{H}}\left(  \gamma\right)
f\right\rangle _{L^{2}(\mathbb{H})}\right\vert ^{2}\\
&  =%
%TCIMACRO{\dsum \limits_{\gamma\in\Gamma}}%
%BeginExpansion
{\displaystyle\sum\limits_{\gamma\in\Gamma}}
%EndExpansion
\left\vert \int_{\left(  0,1\right]  }\left\langle \mathcal{P}\phi
(\lambda),\pi_{\lambda}\left(  \gamma\right)  \mathcal{P}f(\lambda
)\right\rangle _{HS}\left\vert \lambda\right\vert d\lambda\right\vert ^{2}\\
&  =%
%TCIMACRO{\dsum \limits_{\gamma_{1}\in\Gamma_{1}}}%
%BeginExpansion
{\displaystyle\sum\limits_{\gamma_{1}\in\Gamma_{1}}}
%EndExpansion%
%TCIMACRO{\dsum \limits_{j\in\mathbb{Z}}}%
%BeginExpansion
{\displaystyle\sum\limits_{j\in\mathbb{Z}}}
%EndExpansion
\left\vert \int_{\left(  0,1\right]  }e^{2\pi ij\lambda}\left\langle
\mathcal{P}\phi(\lambda),\pi_{\lambda}\left(  \gamma_{1}\right)
\mathcal{P}f(\lambda)\right\rangle _{HS}\left\vert \lambda\right\vert
d\lambda\right\vert ^{2}.
\end{align*}
Since $\left\{  e^{2\pi ij\lambda}:j\in%
%TCIMACRO{\U{2124} }%
%BeginExpansion
\mathbb{Z}
%EndExpansion
\right\}  $ forms a Parseval frame for $L^{2}\left(  \left(  0,1\right]
\right)  ,$ we have
\begin{align*}
&
%TCIMACRO{\dsum \limits_{\gamma\in\Gamma}}%
%BeginExpansion
{\displaystyle\sum\limits_{\gamma\in\Gamma}}
%EndExpansion
\left\vert \left\langle \phi,L_{\mathbb{H}}\left(  \gamma\right)
f\right\rangle _{L^{2}(\mathbb{H})}\right\vert ^{2}\\
&  =%
%TCIMACRO{\dsum \limits_{\gamma_{1}\in\Gamma_{1}}}%
%BeginExpansion
{\displaystyle\sum\limits_{\gamma_{1}\in\Gamma_{1}}}
%EndExpansion
\int_{\left(  0,1\right]  }\left\vert \left\langle \mathcal{P}\phi
(\lambda),\pi_{\lambda}\left(  \gamma_{1}\right)  \mathcal{P}\left(  f\right)
(\lambda)\left\vert \lambda\right\vert ^{1/2}\right\rangle _{HS}\right\vert
^{2}\left\vert \lambda\right\vert d\lambda.
\end{align*}
Let $\mathcal{P}\left(  f\right)  (\lambda)\left\vert \lambda\right\vert
^{1/2}=\left\vert \lambda\right\vert ^{1/2}\mathbf{w}_{\lambda}\otimes
\mathbf{u}_{\lambda}\in L^{2}\left(
%TCIMACRO{\U{211d} }%
%BeginExpansion
\mathbb{R}
%EndExpansion
\right)  \otimes\mathbf{u}_{\lambda}$ a.e. Then
\begin{equation}%
%TCIMACRO{\dsum \limits_{\gamma\in\Gamma}}%
%BeginExpansion
{\displaystyle\sum\limits_{\gamma\in\Gamma}}
%EndExpansion
\left\vert \left\langle \phi,L_{\mathbb{H}}\left(  \gamma\right)
f\right\rangle _{L^{2}(\mathbb{H})}\right\vert ^{2}=\int_{\left(  0,1\right]
}%
%TCIMACRO{\dsum \limits_{\gamma_{1}\in\Gamma_{1}}}%
%BeginExpansion
{\displaystyle\sum\limits_{\gamma_{1}\in\Gamma_{1}}}
%EndExpansion
\left\vert \left\langle s\left(  \lambda\right)  ,\pi_{\lambda}\left(
\gamma_{1}\right)  \left(  \left\vert \lambda\right\vert ^{1/2}\mathbf{w}%
_{\lambda}\right)  \right\rangle _{L^{2}\left(
%TCIMACRO{\U{211d} }%
%BeginExpansion
\mathbb{R}
%EndExpansion
\right)  }\right\vert ^{2}\left\vert \lambda\right\vert d\lambda\label{line2}%
\end{equation}
where $\mathcal{P}\phi(\lambda)=\mathbf{s}_{\lambda}\otimes\mathbf{u}%
_{\lambda}.$ Notice that by definition $\pi_{\lambda}\left(  \gamma
_{1}\right)  \mathbf{f}=\exp\left(  2\pi i\lambda jt\right)  \mathbf{f}\left(
t-k\right)  $ where $\left(  \lambda j,k\right)  \in\lambda%
%TCIMACRO{\U{2124} }%
%BeginExpansion
\mathbb{Z}
%EndExpansion
\times%
%TCIMACRO{\U{2124} }%
%BeginExpansion
\mathbb{Z}
%EndExpansion
$ with $\lambda\in\left(  0,1\right]  .$ By the density condition given in
Lemma \ref{density}, it is possible to find $\mathbf{v}_{\lambda}$ such that
the system
\[
\left\{  \pi_{\lambda}\left(  \gamma_{1}\right)  \left(  \mathbf{v}_{\lambda
}\right)  :\gamma_{1}\in\Gamma_{1}\right\}
\]
is a Parseval frame in $L^{2}\left(
%TCIMACRO{\U{211d} }%
%BeginExpansion
\mathbb{R}
%EndExpansion
\right)  $ for a.e. $\lambda\in\left(  0,1\right]  .$ So let us suppose that
we pick $f\in\mathbf{H}$ $^{+}$ so that
\[
\mathcal{P}\left(  f\right)  (\lambda)=\left\vert \lambda\right\vert
^{-1/2}\mathbf{v}_{\lambda}\otimes\mathbf{u}_{\lambda}.
\]
We will later see that $f$ is indeed square-integrable. Coming back to
(\ref{line2}),
\begin{align*}%
%TCIMACRO{\dsum \limits_{\gamma\in\Gamma}}%
%BeginExpansion
{\displaystyle\sum\limits_{\gamma\in\Gamma}}
%EndExpansion
\left\vert \left\langle \phi,L_{\mathbb{H}}\left(  \gamma\right)
f\right\rangle _{L^{2}(\mathbb{H})}\right\vert ^{2}  &  =\int_{\left(
0,1\right]  }%
%TCIMACRO{\dsum \limits_{\gamma_{1}\in\Gamma_{1}}}%
%BeginExpansion
{\displaystyle\sum\limits_{\gamma_{1}\in\Gamma_{1}}}
%EndExpansion
\left\vert \left\langle \mathbf{s}_{\lambda},\pi_{\lambda}\left(  \gamma
_{1}\right)  \left(  \mathbf{v}_{\lambda}\right)  \right\rangle _{L^{2}\left(
%
%TCIMACRO{\U{211d} }%
%BeginExpansion
\mathbb{R}
%EndExpansion
\right)  }\right\vert ^{2}\left\vert \lambda\right\vert d\lambda\\
&  =\int_{\left(  0,1\right]  }\left\Vert \mathcal{P}\phi(\lambda)\right\Vert
_{HS}^{2}\left\vert \lambda\right\vert d\lambda\\
&  =\left\Vert \phi\right\Vert _{L^{2}(\mathbb{H})}^{2}.
\end{align*}
To prove that $\left\Vert f\right\Vert _{L^{2}(\mathbb{H})}^{2}=\frac{1}{2},$
we appeal to Lemma \ref{volume} and we obtain
\begin{align*}
\left\Vert \left\vert \lambda\right\vert ^{-1/2}\mathbf{v}_{\lambda
}\right\Vert _{L^{2}\left(
%TCIMACRO{\U{211d} }%
%BeginExpansion
\mathbb{R}
%EndExpansion
\right)  }^{2}  &  =\left\vert \lambda\right\vert ^{-1}\left\Vert
\mathbf{v}_{\lambda}\right\Vert _{L^{2}\left(
%TCIMACRO{\U{211d} }%
%BeginExpansion
\mathbb{R}
%EndExpansion
\right)  }^{2}\\
&  =\left\vert \lambda\right\vert ^{-1}\mathrm{vol}\left(  \lambda%
%TCIMACRO{\U{2124} }%
%BeginExpansion
\mathbb{Z}
%EndExpansion
\times%
%TCIMACRO{\U{2124} }%
%BeginExpansion
\mathbb{Z}
%EndExpansion
\right) \\
&  =1.
\end{align*}
The above holds for almost every $\lambda\in\left(  0,1\right]  $. Next, since
$\left\Vert \mathcal{P}f(\lambda)\right\Vert _{HS}^{2}=1\ $we obtain%
\[
\left\Vert f\right\Vert _{L^{2}(\mathbb{H})}^{2}=\int_{0}^{1}\left\vert
\lambda\right\vert d\lambda=\frac{1}{2}.
\]

\end{proof}

Similarly, we also have the following lemma

\begin{lemma}
The representation $\left(  L_{\mathbb{H}}|\Gamma,\mathbf{H}\text{ }%
^{-}\right)  $ is cyclic. Moreover $\mathbf{H}^{-}$ admits a Parseval frame of
the type $L_{\mathbb{H}}\left(  \Gamma\right)  h$ and
\[
\left\Vert h\right\Vert _{L^{2}(\mathbb{H})}^{2}=\frac{1}{2}.
\]

\end{lemma}

\begin{lemma}
\label{CurZ} Let $\mathbf{H=H}^{+}\oplus\mathbf{H}^{-}.$ Then there exists an
orthonormal basis of the type $L_{\mathbb{H}}\left(  \Gamma\right)  f$ for
$\mathbf{H.}$
\end{lemma}

We remark that in general the direct sum of two Parseval frames in $H,$ and
$K$ is not an even a Parseval frame for the space $H\oplus K$ , unless the
ranges of the coefficients operators are orthogonal to each other. A proof of Lemma \ref{CurZ} is given by Currey and Mayeli in
\cite{Currey}, where they show how to put together $f$ and $h$ in order to
obtain an orthonormal basis for $\mathbf{H}$. Now, we are in a good position
to obtain a decomposition of the left regular representation of the
time-frequency group. First, let us define
\[
K=\left\{  \left[
\begin{array}
[c]{ccc}%
1 & 0 & l\\
0 & 1 & k\\
0 & 0 & 1
\end{array}
\right]  :\left(  k,l\right)  \in%
%TCIMACRO{\U{2124} }%
%BeginExpansion
\mathbb{Z}
%EndExpansion
^{2}\right\}  .
\]
It is easy to see that $K$ is an abelian subgroup of $\Gamma.$ Let $L$ be the
left regular representation of $\Gamma.$ 
\begin{theorem}
A direct integral decomposition of $L$ is obtained as follows
\begin{equation}
\int_{\left[  -1,1\right]  }^{\oplus}\pi_{\lambda}|\Gamma\text{ }\left\vert
\lambda\right\vert d\lambda\cong\int_{\left[  -1,1\right]  }^{\oplus}%
\int_{\left[  0,\left\vert \lambda\right\vert \right)  }^{\oplus}%
\mathrm{Ind}_{K}^{\Gamma}\left(  \chi_{\left(  \left\vert \lambda\right\vert
,t\right)  }\right)  \left\vert \lambda\right\vert dtd\lambda\label{dec}%
\end{equation}
acting in the Hilbert space
\[
\int_{\left[  -1,1\right]  }^{\oplus}\int_{\left[  0,\left\vert \lambda
\right\vert \right)  }^{\oplus}l^{2}\left(
%TCIMACRO{\U{2124} }%
%BeginExpansion
\mathbb{Z}
%EndExpansion
\right)  \text{ }\left\vert \lambda\right\vert dtd\lambda
\]

\end{theorem}

\begin{proof}
First, let us define $R_{f}:\mathbf{H\rightarrow}$ $l^{2}\left(
\Gamma\right)  $ such that $R_{f}g\left(  \gamma\right)  =\left\langle
g,L\left(  \gamma\right)  f\right\rangle .$ Using Lemma \ref{CurZ} (see
\cite{Currey} also), we construct a vector $f\in\mathbf{H}$ such that $R_{f}$
is an unitary. As a result, the operator $R_{f}\circ\mathcal{P}^{-1}$ is
unitary and
\begin{equation}
\left(  R_{f}\circ\mathcal{P}^{-1}\right)  \mathcal{\circ}\left(
\int_{\left[  -1,1\right]  }^{\oplus}\pi_{\lambda}|\Gamma\text{ }\left\vert
\lambda\right\vert d\lambda\right)  \left(  \cdot\right)  \circ\left(
\mathcal{P\circ}R_{f}^{-1}\right)  =L\left(  \cdot\right)  \label{first}%
\end{equation}
where
\[
\pi_{\lambda}\left(  \left[
\begin{array}
[c]{ccc}%
1 & m & l\\
0 & 1 & k\\
0 & 0 & 1
\end{array}
\right]  \right)  f\left(  t\right)  =e^{2\pi il\lambda}e^{2\pi ik\lambda
t}f\left(  t-m\right)  .
\]
Thus,
\[
\int_{\left[  -1,1\right]  }^{\oplus}\pi_{\lambda}|\Gamma\text{ }\left\vert
\lambda\right\vert d\lambda\cong L.
\]
Notice that $\left(  \pi_{\lambda}|\Gamma,L^{2}\left(  \mathbb{R}\right)
\right)  $ is not an irreducible representation. However, we may use Baggett's
decomposition given in \cite{Baggett}. For each $\lambda\in\left[
-1,1\right]  ,$ the representation $\pi_{\lambda}$ is decomposed into its
irreducible components as follows:
\begin{equation}
\pi_{\lambda}|\Gamma\cong\int_{\left[  0,\left\vert \lambda\right\vert \right)
}^{\oplus}\mathrm{Ind}_{K}^{\Gamma}\left(  \chi_{\left(  \left\vert
\lambda\right\vert ,t\right)  }\right)  dt\label{second}%
\end{equation}
where
\[
\chi_{\left(  \left\vert \lambda\right\vert ,t\right)  }\left(  \left[
\begin{array}
[c]{ccc}%
1 & 0 & l\\
0 & 1 & k\\
0 & 0 & 1
\end{array}
\right]  \right)  =\exp\left(  2\pi i\left(  \left\vert \lambda\right\vert
l+tk\right)  \right)  .
\]
Combining (\ref{first}) with (\ref{second}), we obtain the decomposition given
in (\ref{dec}).
\end{proof}
It is worth noticing that in the case where $\alpha\in%
%TCIMACRO{\U{211d} }%
%BeginExpansion
\mathbb{R}
%EndExpansion
-%
%TCIMACRO{\U{211a} }%
%BeginExpansion
\mathbb{Q}
%EndExpansion
,$ that the group $G$ is a non-type I group. Moreover, it is well-known  that the left regular
representation of $G$ is a non-type I representation. In fact, (see\cite{Kan}) the left regular
representation of this group is type II. In this case, in order to obtain a
useful decomposition of the left regular representation, it is better to
consider a new type of dual. Let us recall the following well-known facts (see Section $3.4.2$ \cite{hartmut}). Let
$G$ be a locally compact group. Let $\breve{G}$ be the collection of all
quasi-equivalence classes of primary representations of $G,$ and let $\pi$ be
a unitary representation of $G$ acting in a Hilbert space $\mathcal{H}_{\pi}.$
Essentially, there exists a unique way to decompose the representation $\pi$
into primary representations such that the center of the commuting algebra of
the representation is decomposed as well. This decomposition is known as the
\textbf{central decomposition}. More precisely, there exist

\begin{enumerate}
\item A standard measure $\nu_{\pi}$ on the quasi-dual of $G:\breve{G}$

\item A $\nu_{\pi}$-measurable field of representations $\left(  \rho
_{t}\right)  _{t\in\breve{G}}$

\item A unitary operator
\[
\mathfrak{R}:\mathcal{H}_{\pi}\rightarrow\int_{\breve{G}}^{\oplus}%
\mathcal{H}_{\pi,t}\text{ }d\nu_{\pi}\left(  t\right)
\]
such that
\[
\mathfrak{R}\pi\left(  \cdot\right)  \mathfrak{R}^{-1}=\int_{\breve{G}%
}^{\oplus}\rho_{t}\left(  \cdot\right)  \text{ }d\nu_{\pi}\left(  t\right)  .
\]
Moreover, letting $Z\left(  \mathcal{C}\left(  \pi\right)  \right)  $ be the
center of the commuting algebra of the representation $\pi,$ then
\[
\mathfrak{R}\left(  Z\left(  \mathcal{C}\left(  \pi\right)  \right)  \right)
\mathfrak{R}^{-1}=\int_{\breve{G}}^{\oplus}%
%TCIMACRO{\U{2102} }%
%BeginExpansion
\mathbb{C}
\cdot\mathbf{1}_{\mathcal{H}_{\pi,t}}\text{ }d\nu_{\pi}\left(  t\right)
.
\]

\end{enumerate}

The importance of the central decomposition is illustrated through the
following facts. Let $\pi,\theta$ be representations of a locally compact group and let $\nu_{\pi}$ and
$\nu_{\theta}$ be the measures underlying the respective central
decompositions. Then $\pi$ is quasi-equivalent to a subrepresentation of
$\theta$ if and only if $\nu_{\pi}$ is absolutely continuous with respect to
$\nu_{\theta}.$ In particular, $\pi$ is quasi-equivalent to $\theta$ if and
only if the measures $\nu_{\pi}$ and $\nu_{\theta}$ are equivalent.
Furthermore, the representations $\nu_{\pi}$ and $\nu_{\theta}$ are disjoint
if and only if the measures $\nu_{\pi}$ and $\nu_{\theta}$ are disjoint
measures. Since the central decomposition provides information concerning the commuting algebra of the left regular representation $L$, and because such information is crucial in the classification of admissible representations of $\Gamma$ then  it is important to mention the following. 

\begin{theorem}\label{central}
$L\cong\int_{\left(  0,1\right]  }^{\oplus}\pi_{\lambda}|\Gamma\otimes
\mathbf{1}_{%
%TCIMACRO{\U{2102} }%
%BeginExpansion
\mathbb{C}
%EndExpansion
^{2}}$ $\left\vert \lambda\right\vert d\lambda$ and this decomposition is the
central decomposition of the left regular representation of $G.$
\end{theorem}

\begin{proof}
We have already seen that
\[
L\cong\int_{\left[  -1,1\right]  }^{\oplus}\pi_{\lambda}|\Gamma\left\vert
\lambda\right\vert d\lambda
\]
and
\[
\pi_{\lambda}|\Gamma\cong\int_{\left[  0,\left\vert \lambda\right\vert
\right)  }^{\oplus}\mathrm{Ind}_{K}^{\Gamma}\left(  \chi_{\left(  \left\vert
\lambda\right\vert ,t\right)  }\right)  \text{ }dt.
\]
Therefore given $\ell_{1},\ell_{2}\in\left[  -1,1\right]  -\left\{  0\right\}
,$ if $\left\vert \ell_{1}\right\vert =\left\vert \ell_{2}\right\vert $ then
$\pi_{\ell_{1}}|\Gamma\cong\pi_{\ell_{2}}|\Gamma.$ As a result, the left
regular representation of $G$ is equivalent to
\begin{equation}
\int_{\left(  0,1\right]  }^{\oplus}\pi_{\lambda}|\Gamma\otimes\mathbf{1}_{%
%TCIMACRO{\U{2102} }%
%BeginExpansion
\mathbb{C}
%EndExpansion
^{2}}\left\vert \lambda\right\vert d\lambda\label{factor}%
\end{equation}
Also, it is well-known that the representation $\pi_{\lambda}|\Gamma$ is a primary or
a factor representation of $G$ whenever $\lambda$ is irrational (see Page $127$ of \cite{Folland}). So, the
decomposition given in (\ref{factor}) is indeed a central decomposition of $L$. This
completes the proof. 
\end{proof}


\begin{thebibliography}{99}                                                                                              
\bibitem {Baggett}L. Baggett, Processing a radar signal and representations of
the discrete Heisenberg group, Colloq. Math. 60/61 (1990) 195--203.

\bibitem {Bownik}M. Bownik, The structure of shift-modulation invariant
spaces: the rational case, J. Funct. Anal. 244 (2007), 172-219.


\bibitem {Corwin}L. Corwin, F. Greenleaf, Representations of nilpotent Lie
groups and their applications. Part I. Basic theory and examples, Cambridge
Studies in Advanced Mathematics, 18. Cambridge University Press, Cambridge, 1990.

\bibitem {Currey}B. Currey, A. Mayeli, A Density Condition for Interpolation
on the Heisenberg Group, Rocky Mountain J. Math.,  Volume 42, Number 4 (2012), 1135-1151.
 

\bibitem {Folland}G. Folland, A course in abstract harmonic analysis (1995)
(Boca Raton, Florida: CRC Press)
\bibitem{Folland1} G. Folland, The abstruse meets the applicable: some aspects of time-frequency analysis, Proc. Indian Acad. Sci. Math. Sci. 116 (2006), no. 2, 121-136.
\bibitem {hartmut}H. F\"{u}hr, Abstract harmonic analysis of continuous
wavelet transforms. Lecture Notes in Mathematics, 1863. Springer-Verlag,
Berlin, 2005.

\bibitem {Gabor} D. Gabor, Theory of Communication,\ J. Inst. Elec. Eng., (London), 93, pp. 429-457, 1946.

\bibitem {Unified}G. Cariolaro, Unified Signal Theory, 1st Edition, 2011, Springer.

\bibitem {Grog}K. Gr\"{o}chenig, Foundations of Time-Frequency Analysis,
Birkh\"{a}user, Boston, 2001.

\bibitem {Han}D. Han and Y. Wang, Lattice tiling and the Weyl-Heisenberg
frames, GAFA 11 (2001), no. 4, 742--758.

\bibitem {Wilson}E. Hern\'{a}ndez, H. \v{S}iki\'{c}, G. Weiss, E. Wilson, The
Zak transform(s), Wavelets and multiscale analysis, 151--157, Appl. Numer.
Harmon. Anal., Birkh\"{a}user/Springer, New York, 2011.
\bibitem{Kan} E. Kaniuth, Der Typ der regulären Darstellung diskreter Gruppen, Math. Ann., 182 (1969), pp. 334--339
\bibitem {Lipsman1}A. Kleppner, R. Lipsman, The Plancherel formula for group
extensions. I, II. Ann. Sci. \'{E}cole Norm. Sup. (4) 5 (1972), 459--516.

\bibitem {Lipsman2}R. Lipsman, Group representations. A survey of some current
topics. Lecture Notes in Mathematics, Vol. 388. Springer-Verlag, Berlin-New York, 1974.

\bibitem {Oussa}V. Oussa, Bandlimited Spaces on Some 2-step Nilpotent Lie
Groups With One Parseval Frame Generator, to appear in Rocky Mountain Journal
of Mathematics, 2012.

\bibitem {Pfander}G. Pfander, P. Rashkov, Y. Wang, A Geometric Construction of
Tight Gabor Frames with Multivariate Compactly Supported Smooth Windows.

\bibitem {wolf}J. Wolf, Harmonic analysis on commutative spaces. Mathematical
Surveys and Monographs, 142. American Mathematical Society, Providence, RI, 2007.
\end{thebibliography}
\end{document}